\documentclass[12pt,a4paper]{article}
\usepackage{enumerate}
\usepackage[active]{srcltx}
\usepackage{color}
\usepackage[dvips]{graphicx}
\usepackage{multirow}
\usepackage{amsmath}
\usepackage{amsfonts,latexsym}
\usepackage{amsthm}
\usepackage{verbatim}
\usepackage{subfigure}
\usepackage{amssymb}
\usepackage{epsfig}

\newtheorem{theorem}{Theorem}[section]
\newtheorem{corollary}[theorem]{Corollary}
\newtheorem{lemma}[theorem]{Lemma}
\newtheorem{proposition}[theorem]{Proposition}

\theoremstyle{definition}

\theoremstyle{remark}
\newtheorem{remark}[theorem]{Remark}

\numberwithin{equation}{section}

\begin{document}

\markboth{J.P. Borgna, M. De Leo, C. S\'anchez de la Vega, D.
Rial}{Lie-Trotter method for abstract semilinear \\ evolution equations.}

\title{High order time--splitting methods for irreversible equations}

\author{
Mariano De Leo\thanks{Instituto de Ciencias, Universidad Nacional
de General Sarmiento, J.M. Guti\'errez 1150 (1613) Los Polvorines,
Buenos Aires,
Argentina. Email: {\em mdeleo@ungs.edu.ar}}\,,\; 
Diego Rial\thanks{IMAS - CONICET and Departamento de Matem\'atica, Facultad de
Ciencias Exactas y Naturales, Universidad de Buenos Aires, Ciudad
Universitaria, Pabell\'on I (1428) Buenos Aires, Argentina. Email:
{\em drial@dm.uba.ar}}\,,\; \\
Constanza S\'anchez de la Vega\thanks{IMAS - CONICET and Instituto de Humanidad, Universidad Nacional de General Sarmiento, J.M. Guti\'errez 1150 (1613) Los Polvorines, Buenos Aires, Argentina. Email: {\em csfvega@dm.uba.ar}}
}

\maketitle
\begin{abstract}
{In this work, high order splitting methods of integration without negative steps are shown which can be used 
in irreversible problems, like reaction--difussion or complex Guinzburg--Landau equations.
The methods consist in a suitable affine combinations of Lie--Tortter schemes with different positive steps.
The number of basic steps for these methods grows quadratically with the order, while for symplectic methods, the growth is exponential.
Furthermore, the calculations can be performed in parallel, so that the computation time can be significantly reduced
using multiple processors.
Convergence results of these methods are proved for a large kind of semilinear problems, 
that includes reaction-difussion systems and dissipative perturbation of Hamiltonian systems.}
{splitting methods, irreversible dynamics, high order method} 
\end{abstract}

\section{Introduction}

The goal of the present article is to derive arbitrary order splitting integrators 
for irreversible problems. We are mainly interested in dissipative 
pseudo-differentiable problems which cannot be solved neither by lines methods nor by usual
splitting integrators with negative steps.
In order to avoid negative steps, symplectic methods with complex steps are proposed in 
the literature, but in this case analytic properties on the operators are required. 
These assumptions on the operators restrict the application of this kind of methods
to reaction--diffusion type problems.

In this article we obtain integrators that, at the same time, avoid the use of negative 
steps and do not require special assumptions on the operator, as well as they exploit
the simplicity of the decomposition of the original problem. These methods can also be
applied to problems with nonlocal nonlinearities as it is shown below.  
It is possible to build arbitrary high order integrators for which the number of basic 
steps is lower than previous symplectic methods. Moreover, these methods can naturally
be parallelized.
In this work, we present a rigorous proof of the convergence of the proposed methods, 
and we also test their performance in several examples of interest.

We study the initial value problem
\begin{align}
\begin{cases}
\label{eq: ut=Au}
\partial_{t}u=A_{0}\,u+A_{1}(u),\\
u(0)=u_{0},
\end{cases}
\end{align}
where $A_{0}$ is a linear closed operator densely defined in  $D(A_{0})\subset\mathsf{H}$,
$\mathsf{H}$ is a Hilbert space, which generates a quasicontraction semi-group of operators.
We assume that the nonlinear term $A_{1}:\mathsf{H}\to\mathsf{H}$ is a smooth mapping with $A_{1}(0)=0$.
In many problems of interest, the partial equations
\begin{subequations}
\label{eq: flujos parciales}
\begin{align}
\label{eq: ut=A0u}
\partial_{t}u=&\,A_{0}\,u,\\
\label{eq: ut=A1u}
\partial_{t}u=&\,A_{1}(u),
\end{align}
\end{subequations}
can be easily solved either analytically or numerically, which enable to find approximated solutions of the problem \eqref{eq: ut=Au}
applying in turn the flows $\phi_0$ and $\phi_1$ associated to each partial problem \eqref{eq: ut=A0u} and \eqref{eq: ut=A1u} respectively.

There exist many numerical integration methods for \eqref{eq: ut=Au}
based on splitting methods, the most known are the Lie--Trotter and Strang methods defined by
\begin{align*}
\Phi_{\mathrm{Lie}}(h,u)=&\,\phi_{1}(h,\phi_{0}(h,u)), \\
\Phi_{\mathrm{Strang}}(h,u)=&\,\phi_{0}(h/2,\phi_{1}(h,\phi_{0}(h/2,u))),
\end{align*}
where $h$ is the time step of the numerical integration.
It can be proved that $\Phi_{\mathrm{Lie}}$ has order $1$ and $\Phi_{\mathrm{Strang}}$ has order $2$,
where the order $q$ represents 
the greatest natural number such that the truncation error between the real flow $\phi$ 
of the equation \eqref{eq: ut=Au} and the numerical method $\Phi$ satisfies
\[
\|\phi(h,u)-\Phi(h,u)\|_{\mathsf{H}}\leq C(u) h^{q+1}
\]
for $0<h<h_{*}$. 

A highly known example of problem \eqref{eq: ut=Au} is the nonlinear Schr\"odinger equation (NLS) 
\begin{align}
\label{eq: NLS}
\partial_{t}u=\mathrm{i}\Delta u+\mathrm{i}|u|^{2}u,
\end{align}
where the partial flows associated to each term of the equation are given by
\begin{align*}
\phi_{0}(t,u)=&\,\exp(\mathrm{i}t\Delta)u,\\
\phi_{1}(t,u)=&\,\exp(\mathrm{i}t|u|^{2})u,
\end{align*}
which represent the evolution of a free particle and self--phase modulation respectively.
This is not exactly the problem we are interested in solving since $A_0$ generates a strongly continuous group of operators, that is we are in the presence of a reversible system.
In \cite{Ruth1983}, \cite{Neri1987} and \cite{Yoshida1990}, the authors present numerical integrators for Hamiltonian systems of order $q=3,4,2n$ respectively,
which are known as symplectic integrators. The general form of this methods is the following:
\begin{align}
\label{eq: split}
\Phi_{Sym}(h)=
\phi_{1}(b_{m}h)\circ\phi_{0}(a_{m}h)\circ\cdots\circ\phi_{1}(b_{1}h)\circ\phi_{0}(a_{1}h),
\end{align}
with $a_{1}+\cdots+a_{m}=b_{1}+\cdots+b_{m}=1$.
In the pioneering work \cite{Ruth1983}, a symplectic operator $\Phi_{Sym}$ of order $3$ is presented,
taking $a_{1}=7/24$, $a_{2}=3/4$, $a_{3}=-1/24$ and $b_{1}=2/3$, $b_{2}=-2/3$, $b_{3}=1$.
In \cite{Neri1987} a symplectic operator of order $4$ is considered, where
\begin{align*}
a_{1}=&\,a_{4}=\frac{1}{2(2-2^{1/3})},\;a_{2}=a_{3}=-\frac{2^{1/3}-1}{2(2-2^{1/3})}, \\
b_{1}=&\,b_{3}=\frac{1}{2-2^{1/3}},\; b_{2}= -\frac{2^{1/3}}{2-2^{1/3}},\; b_{4}=0.
\end{align*}
In \cite{Yoshida1990}, Yoshida presents a systematic way to obtain integrators of arbitrary even order,
based on the Baker--Campbell--Hausdorff formula. These integrators can be set inductively
\[
\Phi_{Sym,2n+2}(h)=\Phi_{Sym,2n}(z_{1}h)\circ\Phi_{Sym,2n}(z_{0}h)\circ\Phi_{Sym,2n}(z_{1}h),
\]
with $z_{0}+2z_{1}=1$ and $z_{0}^{2n+1}+z_{1}^{2n+1}=0$. The total number of steps of the method of order
$q=2n$ is $S_{T}=3^{n}$.
Nevertheless, for order $q=6,8$ there can be shown symplectic integrators with $8$ and $16$ steps respectively.


In the last years, many authors started the rigorous study of the convergence of the symplectic methods
applied to Hamiltonian systems in infinite dimension.
In \cite{Besse2002} the NLS problem given by \eqref{eq: NLS} in dimension $2$ is considered and it is proved the convergence of the Lie--Trotter and Strang methods in $L^{2}(\mathbb{R}^{2})$ with order $1$ and $2$ respectively
(see also \cite{Descombes2010} and \cite{Descombes2013}).
In \cite{Lubich2008} and \cite{Gauckler2011} similar results are proved for the Gross--Pitaevskii equation given by:
\[
\mathrm{i}\partial_{t}u=-\Delta u+|x|^{2}u+|u|^{2}u,
\]
In both cases, the solutions are needed to be differentiable with respect to time,
and therefore initial data in $D(A_{0}^{k})$ 
is considered, where $A_{0}$ is the corresponding differential operator.

The symplectic methods with order $q>2$ require some step to be negative (see \cite{Goldman1996}),
inhibiting its application to irreversible problems.
In \cite{Castella2009}, the authors develop splitting methods for irreversible problems,
that use complex time steps having positive real part: going to the complex plane allows to considerably increase the 
accuracy, while keeping small time steps. The total number of steps using the so called triple jump method of order 
$q=2n$ is $S_{T}=3^{n-1}$ for order not greater than $8$ and for the quadruple jump method is 
$S_T=4\times 3^{n-2}$ for order not greater than $12$. Finally we recall that the rigorous approach given in this 
article is based upon the results for linear operators given in \cite{Hansen2009} while the nonlinear problem is only 
formally discussed.

Since our interest is focused on  irreversible pseudo--differential problems, the paradigmatic example we have in mind is the regularized cubic Schr\"odinger equation:
\begin{equation} \label{eq: pseudo diff Schroedinger}
\partial_{t}u=\mathrm{i}\Delta u-(-\Delta)^{\beta}u+\mathrm{i}|u|^{2}u,
\end{equation}
where $0<\beta<1$. 
It is natural to split the problem into the linear equation $\partial_{t}u=\mathrm{i}\Delta u-(-\Delta)^{\beta}u$
and the ordinary differential equation system given by $\dot{u}=\mathrm{i}|u|^{2}u$,
where the linear problem is ill-posed for negative times.
Note that the same procedure can be applied to nonlocal nonlinearities like 
convolution potentials as it is done in example 4.3 below (see also example 4.1 in \cite{Borgna2015}). 
Since $\mathrm{i}\Delta-(-\Delta)^{\beta}$ is a pseudo--differential operator, it can not be discretized in space
in order to use some method of lines, as Runge--Kutta schemes.
Observe that the strongly continuous semigroup generated by the 
 linear part of equation \eqref{eq: pseudo diff Schroedinger}
can not be extended to an open sector $\{z\in\mathbb{C}:|\arg(z)|<\theta\}$ 
since its spectrum is $\{-\mathrm{i}\lambda-\lambda^{\beta}:\lambda\ge 0\}\not\subseteq
\{\lambda\in\mathbb{C}:\mathrm{arg}|\lambda-\omega|\ge\pi/2+\theta\}$ for any $\omega\in\mathbb{R}$,
contrary to Hille--Yosida--Phillips theorem (see \cite{Reed1975}, theorem X.47b).
Therefore, splitting methods with complex times described in \cite{Castella2009} can not be used.
The case $\beta=1$ corresponds to the complex Ginzburg--Landau equation
(see \cite{Aranson2002} and references there):
\begin{equation}
\label{eq: Ginzburg-Landau}
\partial_{t}u=a\Delta u+b|u|^{2}u,
\end{equation}
where $a,b\in\mathbb{C}$ with $\mathrm{Re}(a)>0$. 
The spectrum of the operator $a\Delta$ is $\sigma(a\Delta)=\{-a\lambda: \lambda\ge 0\}$
and generates a strongly continuous semi-group on the open sector 
$\{z\in\mathbb{C}:|\arg(z)|<\pi/2-|\arg(a)|\}$. In \cite{Castella2009},
it is shown that the arguments of the complex steps grow with the order of the method,
exceeding the value $\pi/2-|\arg(a)|$ for order high enough.
Therefore, among integrators proposed in \cite{Castella2009}, only the low-order methods can be used. 

In this work, we present a family of splitting type methods for arbitrary order with positive time step, 
that exploit the simplicity of the partial flows in non reversible problems.
Here we describe the methods proposed:
given the associated flows $\phi_{0}, \phi_{1}$ of the partial problems, we define the maps
$\Phi^{+}(h)=\phi_{1}(h)\circ\phi_{0}(h)$, 
$\Phi^{-}(h)=\phi_{0}(h)\circ\phi_{1}(h)$ and
$\Phi_{m}^{\pm}(h)=\Phi^{\pm}(h)\circ\Phi_{m-1}^{\pm}(h)$ with $\Phi_{1}^{\pm}=\Phi^{\pm}$, and consider the following methods:
\begin{subequations}
\label{eq: metodos afines}
\begin{align}
\label{eq: metodo afin asimetrico}
\Phi(h)=&\,\sum_{m=1}^{s}\gamma_{m}\Phi_{m}^{\pm}(h/m)&\text{ (asymmetric),}&\\
\label{eq: metodo afin simetrico}
\Phi(h)=&\,\sum_{m=1}^{s}\gamma_{m}
(\Phi_{m}^{+}(h/m)+\Phi_{m}^{-}(h/m))&\text{ (symmetric)}&.
\end{align}
\end{subequations}
We will show below that under appropriated assumptions,
the integrators given by \eqref{eq: metodo afin asimetrico} and
\eqref{eq: metodo afin simetrico} are convergent with order $q$,
if $\gamma=(\gamma_1,\ldots,\gamma_{s})$ satisfies the following conditions
\begin{subequations}
\begin{align}
\label{eq: cond asim}
\begin{split}
1=&\,\gamma_{1}+\gamma_{2}+\cdots+\gamma_{s},\\
0=&\,\gamma_{1}+2^{-k}\gamma_{2}+\cdots+s^{-k}\gamma_{s},\quad 1\leq k\leq q-1,
\end{split}
\end{align}
\begin{align}
\label{eq: cond sim}
\begin{split}
\frac{1}{2}=&\,\gamma_{1}+\gamma_{2}+\cdots+\gamma_{s},\\
0=&\,\gamma_{1}+2^{-2k}\gamma_{2}+\cdots+s^{-2k}\gamma_{s},\quad 1\leq k\leq n-1,
\end{split}
\end{align}
\end{subequations}
respectively, where $2n=q$.
The first method \eqref{eq: metodo afin asimetrico} is the $h$-extrapolation of the first order Lie--Trotter splitting 
method and the second method \eqref{eq: metodo afin simetrico} is the $h^{2}$-extrapolation of the symmetrization 
of this method.
The general extrapolation technique is described in \cite{Hairer1993} and an application of these techniques applied to
classical Hamiltonian systems is shown in \cite{Chin2010}.

The possibility of computing $\Phi_{m}^{\pm}$ simultaneously,
allows to reduce significantly the total time of computation 
using multiple processors.
The total number of steps for \eqref{eq: metodo afin asimetrico} is given by
$S_{T}=2\sum_{\gamma_{m}\neq 0}m$ 
and $S_{T}=4\sum_{\gamma_{m}\neq 0}m$ for \eqref{eq: metodo afin simetrico}.
Neglecting the communication time between the processors, the total time of computation working in parallel,
turns out to be proportional to $S_{P}=2\max\limits_{\gamma_{m}\neq 0} m$ in both cases.
The system \eqref{eq: cond asim} has solution for $s\geq q$, and hence there exist methods of arbitrary order
$q$ with $S_{P}=2q$ and $S_{T}=q(q+1)$.
On the other side, the system \eqref{eq: cond sim} has solution for $s\geq n$, which shows that there exist 
integrators of arbitrary even order $q=2n$ with $S_{P}=q$ and $S_{T}=q(q/2+1)$, using the double of processors.
As it can be seen the minimum number of steps working in parallel for the symmetric
method is smaller than the corresponding one for the asymmetric method.
Also, in the examples considered below, the symmetric method presents less error than the asymmetric method.
These two latter issues pointed out justify the choice of the symmetric method over the asymmetric one. 
Even using one single processor, the total number of steps grows quadratically with the order,
while both methods presented in \cite{Yoshida1990} and \cite{Castella2009} have an exponential growth.

The paper is organized as follows:
In section \ref{sec 2} we give the basic definitions and preliminary results.
We define the stability and uniform stability bounds for an application which extend
the logarithmic norm notion given in 
\cite{Dekker1984}. Following the ideas of \cite{Besse2002}, \cite{Lubich2008} and \cite{Gauckler2011}, we consider a decreasing sequence of dense subspaces where the flows are repeatedly differentiable. 
In section \ref{sec 3} we prove consistency and stability results for the 
methods \eqref{eq: metodos afines}, from where we deduce the convergence in the standard way.
In section \ref{sec 4} we give several examples of the application of the methods to initial value problems
for ODE's and irreversible PDE's.

\section{Notation and preliminary results}
\label{sec 2}
From now on, we will denote $\phi$ the flow of the equation \eqref{eq: ut=Au}, $\phi_{0}$ and $\phi_{1}$
the flows associated to the respective partial problems $\eqref{eq: ut=A0u}$ and $\eqref{eq: ut=A1u}$.
Also, we will write $\Phi^{\pm}$ the maps defined by $\Phi^{+}(h)=\phi_{1}(h)\circ\phi_{0}(h)$, 
$\Phi^{-}(h)=\phi_{0}(h)\circ\phi_{1}(h)$ and $\Phi_{m}^{\pm}(h)=\Phi^{\pm}(h)\circ\Phi_{m-1}^{\pm}(h)$ 
with $\Phi_{1}^{\pm}=\Phi^{\pm}$.
Finally, we will use the letter $\Phi$ for the numerical integrators given by \eqref{eq: metodo afin asimetrico}
and \eqref{eq: metodo afin simetrico}.

In the next subsections we will give some preliminary results which will be used in section \ref{sec 3}.
Subsection \ref{ssec: combinatoria} provides combinatorial results necessary to prove the consistency in subsection \ref{ssec 31}.
The proof of stability given in subsection \ref{ssec 32} requires the results for stable maps proved in subsection \ref{ssec 22}.
In order to prove theorems \ref{th: consistencia asimetrico} and \ref{th: consistencia simetrico} we establish the concept of compatible flows given in subsection \ref{ssec 23}.

\subsection{Combinatorial results}
\label{ssec: combinatoria}
For a multiindex $\beta=(\beta_{1},\ldots,\beta_{r})\in\mathbb{N}^{r}$, we define
$\beta!=\beta_{1}!\ldots \beta_{r}!$ and $I_{r,k}=\{\beta\in\mathbb{N}^{r}:\beta_{1}+\cdots+\beta_{r}=k\}$ which
 satisfy $\mathbb{N}^{r}=\bigcup_{k=1}^{\infty}I_{r,k}$.

\noindent
\begin{remark}
\label{obs: Irk}
It holds $I_{r,k}=\emptyset$ if $r>k$, 
$I_{k,k}=\{(1,\ldots,1)\}$ and for $r+s\leq k$, 
$I_{r+s,k}=\bigcup_{j=s}^{k-r}I_{r,k-j}\times I_{s,j}$.
\end{remark}
We will need the following lemmas. We will give an outline of the proof of the first lemma and skip the proof of the second one.
\begin{lemma}
\label{le: combinatorio}
Let $q\in\mathbb{N}$, if $\gamma=(\gamma_{1},\ldots,\gamma_{s})$ satisfies the conditions \eqref{eq: cond asim}, then
for $1\leq k\leq q$, it holds that
\begin{align*}
\sum_{m=r}^{s}\binom{m}{r}m^{-k}\gamma_{m}=&\,0,\qquad r=1,\ldots,k-1,\\
\sum_{m=k}^{s}\binom{m}{k}m^{-k}\gamma_{m}=&\,\frac{1}{k!}.
\end{align*}
\end{lemma}

\begin{proof}
We consider the falling factorial $(x)_k=x(x-1)\dots(x-k+1)$,
which is a monic polynomial of degree $k$ such that
$(x)_k=\sum_{j=0}^k S(k,j)x^j.$ Then, for any natural number $m$ satisfying $0\leq m\leq k-1$, we have that $(m)_k=0$ and therefore  $\sum_{j=0}^k S(k,j) m^j=0.$
For the second equality we use that for $1\leq r \leq k-1$
\begin{equation*}
\sum_{m=r}^{s}\binom{m}{r}m^{-k}\gamma_{m}=\frac{1}{r!} \sum_{m=r}^{s} (m)_r m^{-k} \gamma_{m} =- \frac{1}{r!} \sum_{m=1}^{r-1} \sum_{j=0}^{r} S(r,j) m^{j}  \frac{\gamma_{m}}{m^k}=0
\end{equation*}
where we have used the hypothesis on the second equality.
Analogously for the first equality we have:
\begin{equation*}
k!\sum_{m=k}^s \binom{m}{k}m^{-k}\gamma_{m}=\sum_{m=k}^s (m)_k m^{-k}\gamma_{m}= 1 - \sum_{m=1}^{k-1} \left( \frac{\sum_{j=0}^{k} S(k,j) m^{j}}{m^k} \right) \gamma_{m}=1
\end{equation*}
where we have used the hypothesis on the second equality.
\end{proof}

\begin{lemma}
\label{le: combinatorio simetrico}
Let $n\in\mathbb{N}$, if $\gamma=(\gamma_{1},\ldots,\gamma_{s})$ satisfies the conditions \eqref{eq: cond sim}, then
for $1\leq k\leq q=2n$, it holds that
\begin{align*}
\sum_{m=1}^{s}\left[\binom{m}{r}+(-1)^{k+r}\binom{m+r-1}{m-1}\right]
m^{-k}\gamma_{m}=&\,0,\quad r=1,\ldots,k-1,\\
\sum_{m=1}^{s}\left[\binom{m}{k}+\binom{m+k-1}{m-1}\right]
m^{-k}\gamma_{m}=&\,\frac{1}{k!}.
\end{align*}
\end{lemma}
\begin{proof}
The proof is similar to the previous lemma.
\end{proof}

\subsection{Stable maps}
\label{ssec 22}
Let $\mathsf{H}$ be a Hilbert space, and $\varphi:\mathbb{R}_{+}\times \mathsf{H}\to \mathsf{H}$
a continuous map
such that $\varphi(h)=\varphi(h,\cdot):\mathsf{H}\to \mathsf{H}$ is Lipschitz continuous and $\varphi(0)=I$, 
we define
\begin{align*}
\Lambda(\varphi,h)=
\mathop{\sup_{u,u'\in \mathsf{H}}}_{u\neq u'}\frac{\|\varphi(h,u)
-\varphi(h,u')\|_{\mathsf{H}}}{\|u-u'\|_{\mathsf{H}}}.
\end{align*} 
We say that $\varphi$ is stable if 
$ \kappa(\varphi)=\limsup_{h\downarrow 0}h^{-1}(\Lambda(\varphi,h)-1)<\infty$.
For any $\kappa>\kappa(\varphi)$, there exists $h^{*}(\kappa)>0$ such that
\[
\Lambda(\varphi,h)\leq 1+\kappa h\leq \mathrm{e}^{\kappa h},
\]
if $0<h<h_{*}=h_{*}(\kappa)$.
For $\varphi$ a linear flow, $\kappa(\varphi)$ is the logarithmic norm of the generator (see \cite{Dekker1984}).
A map $\varphi$ is called uniformly stable if 
\[
\mu(\varphi)=\limsup_{h\downarrow 0}h^{-1}\Lambda(\varphi-I,h)<\infty.
\]
Since $\Lambda(\varphi,h)\leq 1+\Lambda(\varphi-I,h)$, uniform stability implies stability.
Observe that the family of (uniformly) stable maps is scale-invariant and
if $\varphi_{\lambda}(h,u):=\varphi(\lambda h,u)$ with $\lambda>0$, then
$\kappa(\varphi_{\lambda})=\lambda\kappa(\varphi)$,
$\mu(\varphi_{\lambda})=\lambda\mu(\varphi)$.
If $\varphi$ is a quasicontraction semi-group then $\varphi$ is stable but it is uniformly stable
if and only if the infinitesimal generator is a bounded operator.

\begin{proposition}
\label{pr: estabilidad}
If $\phi_{0},\phi_{1}$ are (uniformly) stable, then the map $\varphi$ defined by
$\varphi(h,u)=\phi_{0}(h,\phi_{1}(h,u))$, is (uniformly) stable and
$\kappa(\varphi)\leq\kappa(\phi_{0})+\kappa(\phi_{1})$
($\mu(\varphi)\leq\mu(\phi_{0})+\mu(\phi_{1})$).
\end{proposition}
\begin{proof}
Since $\Lambda(\varphi,h)\leq\Lambda(\phi_{0},h)\Lambda(\phi_{1},h)$, it follows that
\[
\frac{\Lambda(\varphi,h)-1}{h}\leq\frac{\Lambda(\phi_{0},h)-1}{h}
+\Lambda(\phi_{0},h)\frac{\Lambda(\phi_{1},h)-1}{h},
\]
using that $\Lambda(\phi_{0},h)\to 1$, we get the stability.
Writing $\varphi-I=(\phi_{0}-I)\circ\phi_{1}+\phi_{1}-I$, we have
\[
\Lambda(\varphi-I,h)\leq\Lambda(\phi_{0}-I,h)\Lambda(\phi_{1},h)+\Lambda(\phi_{1}-I,h).
\]
and then $\mu(\varphi)\leq\mu(\phi_{0})+\mu(\phi_{1})$.
\end{proof}

Let $\{\Phi_{m}\}_{1\leq m\leq s}$ be a family of stable maps and $\Phi$ an affine combination, i.e.
$\Phi=\gamma_{1}\Phi_{1}+\cdots+\gamma_{s}\Phi_{s}$ with $\gamma_{1}+\cdots+\gamma_{s}=1$,
it is easy to see that
\[
\Lambda(\Phi,h)\leq\sum_{m=1}^{s}|\gamma_{m}|\Lambda(\Phi_{m},h),
\]
therefore, $\Phi$ is not necessarily a stable map (but it is true for convex combinations). 
We have
\begin{proposition}
\label{pr: combinacion afin}
If $\{\Phi_{m}\}_{1\leq m\leq s}$ is a family of uniformly stable maps, then an affine combination
$\Phi$ is uniformly stable.
\end{proposition}
\begin{proof}
Writing $I=\gamma_{1} I+\cdots+\gamma_{s} I$ and
$\Phi-I=\gamma_{1}(\Phi_{1}-I)+\cdots+\gamma_{s}(\Phi_{s}-I)$, therefore
we get $\mu(\Phi)\leq\sum_{1\leq m\leq s}|\gamma_{m}|\mu(\Phi_{m})$.
\end{proof}

\subsection{Compatible flows}
\label{ssec 23}
Let $\{\mathsf{H}_{k}\}_{k\geq 0}$ be a sequence of Hilbert spaces satisfying
$\mathsf{H}_{k+1}\hookrightarrow\mathsf{H}_{k}$, we define for $k\geq 0$
\[
\mathcal{D}_{k}=
\{f\in C^{\infty}(\mathsf{H}_{k},\mathsf{H}_{0}):\left.f\right|_{H_{k+l}}\in
C^{\infty}(\mathsf{H}_{k+l},\mathsf{H}_{l})\text{ for all } l\geq 0\}.
\]
We can see that if $f\in\mathcal{D}_{k}$ and $g\in\mathcal{D}_{j}$, then $f\circ g\in\mathcal{D}_{j+k}$.
Let $\varphi\in C([0,h_{*})\times\mathsf{H}_{0},\mathsf{H}_{0})$, we say that $\varphi$ is compatible with $\{\mathsf{H}_{k}\}_{k\geq 0}$
if and only if for $l,k\geq 0$, $\varphi\in C^{k,\infty}([0,h_{*})\times \mathsf{H}_{k+l},\mathsf{H}_{l})$.
As an example, let $A:D(A)\to\mathsf{H}$ be a self--adjoint operator, if we take
$\mathsf{H}_{k}=D(A^{k})$ with the inner product
$\langle u,v\rangle_{\mathsf{H}_{k}}=\langle u,v\rangle_{\mathsf{H}}+
\langle A^{k}u,A^{k}v\rangle_{\mathsf{H}}$, we see that $A^{k}\in\mathcal{D}_{k}$.
Assume $\varphi$ is the unitary group with infinitesimal generator $iA$, we have
$\varphi$ is compatible with $\{\mathsf{H}_{k}\}_{k\geq 0}$ and
\[
\frac{\partial^{k}}{\partial h^{k}}\varphi(h,u)=\varphi(h,(iA)^{k}u).
\]
Let $f\in\mathcal{D}_{j}$ and $\varphi$ compatible with $\{\mathsf{H}_{k}\}$,
we have $f\circ\varphi\in C^{k,\infty}([0,h_{*})\times\mathsf{H}_{j+k+l},\mathsf{H}_{l})$. Then 
we define the linear operator $L_{k}[\varphi]:\mathcal{D}_{j}\to\mathcal{D}_{j+k}$ as
\[
(L_{k}[\varphi]f)(u)=\left.\frac{\partial^{k}}{\partial h^{k}}f(\varphi(h,u))\right|_{h=0},
\]
with $u\in \mathsf{H}_{j+k}$.
\begin{lemma}
\label{le: phi psi}
If $\varphi$ and $\psi$ are compatible with $\{\mathsf{H}_{k}\}_{k\geq 0}$, then
$\varphi\circ\psi$ also is compatible with $\{\mathsf{H}_{k}\}_{k\geq 0}$ and satisfies
\begin{align*}
L_{k}[\varphi\circ\psi]=\sum_{j=0}^{k}\binom{k}{j}L_{k-j}[\psi]L_{j}[\varphi].
\end{align*}
\end{lemma}
\begin{proof}
Let $\theta(\tau,\eta,u)=\varphi(\tau,\psi(\eta,u))$,
since $\psi\in C^{k-j,\infty}([0,h_{*})\times\mathsf{H}_{k+l},\mathsf{H}_{j+l})$ and
$\varphi\in C^{j,\infty}([0,h_{*})\times\mathsf{H}_{j+l},\mathsf{H}_{l})$
for $0\leq j\leq k$, then
$\theta\in C^{j,k-j,\infty}([0,h_{*})\times[0,h_{*})\times\mathsf{H}_{k+l},\mathsf{H}_{l})$.
Therefore, $\varphi\circ\psi\in C^{k,\infty}([0,h_{*})\times\mathsf{H}_{k+l},\mathsf{H}_{l})$
and then is compatible with $\{\mathsf{H}_{k}\}_{k\geq 0}$.
Given $f\in\mathcal{D}_{l}$, for any $u\in\mathsf{H}_{k+l}$ it is satisfied 
\begin{align*}
(L_{k}[\varphi\circ\psi]f)(u)=&\,
\sum_{j=0}^{k}\binom{k}{j}
\left.\frac{\partial^{k}}{\partial\eta^{k-j}\partial\tau^{j}}f(\theta(\tau,\eta,u))\right|_{(\tau,\eta)=(0,0)}\\
=&\,\sum_{j=0}^{k}\binom{k}{j}
\left.\frac{\partial^{k-j}}{\partial\eta^{k-j}}(L_{j}[\varphi]f)(\psi(\eta,u))\right|_{\eta=0}\\
=&\,\sum_{j=0}^{k}\binom{k}{j}
(L_{k-j}[\psi]L_{j}[\varphi]f)(u).
\end{align*}
\end{proof}
\begin{lemma}
\label{le: flujo}
If $\varphi$ is a flow, compatible with $\{\mathsf{H}_{k}\}_{k\geq 0}$, then $L_{k}[\varphi]=\left(L_{1}[\varphi]\right)^{k}$.
\end{lemma}
\begin{proof}
The proof is by induction, suppose the result holds for $1\leq j\leq k-1$, using the lemma above we obtain that
\begin{align*}
L_{k}[\varphi\circ\varphi]=&\,2L_{k}[\varphi]+\sum_{j=1}^{k-1}\binom{k}{j}L_{k-j}[\varphi]L_{j}[\varphi]\\
=&\,2L_{k}[\varphi]+\sum_{j=1}^{k-1}\binom{k}{j}\left(L_{1}[\varphi]\right)^{k-j}\left(L_{1}[\varphi]\right)^{j}
=2L_{k}[\varphi]+(2^{k}-2)\left(L_{1}[\varphi]\right)^{k}.
\end{align*}
Since $\varphi(h)\circ\varphi(h)=\varphi(2h)$,
it is obtained that $L_{k}[\varphi\circ\varphi]=2^{k}L_{k}[\varphi]$, which implies the result for $j=k$.
\end{proof}

\section{Convergence}
\label{sec 3}
\subsection{Consistency}
\label{ssec 31}

The next two theorems ensures consistency results for the schemes given by \eqref{eq: metodo afin asimetrico} and
\eqref{eq: metodo afin simetrico}, when the coefficients of the affine combination that defines the methods $\Phi$ satisfy the algebraic conditions \eqref{eq: cond asim} and \eqref{eq: cond sim}, respectively.

Let $\{\mathsf{H}_{k}\}_{k\geq 0}$ be a sequence of Hilbert spaces satisfying
$\mathsf{H}_{k+1}\hookrightarrow\mathsf{H}_{k}$. We will assume that the flow $\phi$ associated to \eqref{eq: ut=Au} and the partial
flows $\phi_{0}$ and $\phi_{1}$ are compatible with $\{\mathsf{H}_{k}\}_{k\geq 0}$.
We have the following consistency results:
\begin{theorem}[Asymmetric case]
\label{th: consistencia asimetrico}
For any $q\in\mathbb{N}$, $\gamma=(\gamma_{1},\ldots,\gamma_{s})$ satisfying \eqref{eq: cond asim} and $u\in\mathsf{H}_{q}$,
the method $\Phi$ given by \eqref{eq: metodo afin asimetrico} satisfies
\[
\frac{\partial^{k}\Phi}{\partial h^{k}}(0,u)=\frac{\partial^{k}\phi}{\partial h^{k}}(0,u),
\]
for $k=0,\ldots,q$.
\end{theorem}

\begin{theorem}[Symmetric case]
\label{th: consistencia simetrico}
For any $n\in\mathbb{N}$, $\gamma=(\gamma_{1},\ldots,\gamma_{s})$ satisfying \eqref{eq: cond sim} and $u\in\mathsf{H}_{q}$
with $q=2n$, the method $\Phi$ given by \eqref{eq: metodo afin simetrico} satisfies
\[
\frac{\partial^{k}\Phi}{\partial h^{k}}(0,u)=\frac{\partial^{k}\phi}{\partial h^{k}}(0,u),
\]
for $k=0,\ldots,q$.
\end{theorem}

\subsubsection{Asymmetric case}
We prove the consistency of method \eqref{eq: metodo afin asimetrico} using
lemma \ref{le: phi psi} and lemma \ref{le: combinatorio}.
\begin{proposition}
\label{pr: L1Lr}
Let $\varphi\in C([0,h_{*})\times\mathsf{H},\mathsf{H})$ be a compatible map with
$\{\mathsf{H}_{k}\}_{k\geq 0}$ satisfying $\varphi(0)=I$. Let
$\varphi_{1}=\varphi$ and $\varphi_{m+1}=\varphi\circ\varphi_{m}$, then
\begin{align*}
L_{k}[\varphi_{m}]=\sum_{r=1}^{k}\binom{m}{r}
\sum_{\beta\in I_{r,k}}\frac{k!}{\beta!}L_{\beta_{1}}[\varphi]\ldots L_{\beta_{r}}[\varphi].
\end{align*}
\end{proposition}
\begin{proof}
Using lemma \ref{le: phi psi}, we get that
\begin{align*}
L_{k}[\varphi_{m+1}]=L_{k}[\varphi_{m}]+L_{k}[\varphi]+
\sum_{j=1}^{k-1}\binom{k}{j}L_{k-j}[\varphi_{m}]L_{j}[\varphi],
\end{align*}
applying induction and using remark \ref{obs: Irk}, we obtain the result.
\end{proof}
\begin{proposition}
\label{pr: dfPhi}
For any $q\in\mathbb{N}$ and $\gamma=(\gamma_{1},\ldots,\gamma_{s})$ satisfying \eqref{eq: cond asim}, the method $\Phi$ given by
\eqref{eq: metodo afin asimetrico} satisfies $L_{k}[\Phi]I=\left(L_{1}[\Phi^{\pm}]\right)^{k}I$, for $k=1,\ldots,q$.
\end{proposition}
\begin{proof}
Since $L_{k}[\Phi]I=\sum_{m=1}^{s}m^{-k}\gamma_{m}L_{k}[\Phi_{m}^{\pm}]I$, using proposition \ref{pr: L1Lr}
we can see that
\begin{align*}
L_{k}[\Phi]I=
\sum_{r=1}^{k}\left(\sum_{m=1}^{s}\binom{m}{r}m^{-k}\gamma_{m}\right)
\sum_{\beta\in I_{r,k}}\frac{k!}{\beta!}L_{\beta_{1}}[\Phi^{\pm}]\ldots L_{\beta_{r}}[\Phi^{\pm}]I,
\end{align*}
from lemma \ref{le: combinatorio}, we get
$L_{k}[\Phi]I=\sum_{\beta\in I_{k,k}}\frac{1}{\beta!}L_{\beta_{1}}[\Phi^{\pm}]\ldots L_{\beta_{r}}[\Phi^{\pm}]I
=\left(L_{1}[\Phi^{\pm}]\right)^{k}I$.
\end{proof}

\begin{proof}{\it (theorem \ref{th: consistencia asimetrico})}
Since $\Phi^{+}=\phi_{1}\circ\phi_{0}$, from lemma \ref{le: phi psi} it holds that
$L_{1}[\Phi^{+}]=L_{1}[\phi_{0}]+L_{1}[\phi_{1}]=L_{1}[\phi]$. In the same way it follows that $L_{1}[\Phi^{-}]=L_{1}[\phi]$.
Using proposition \ref{pr: dfPhi} we obtain that
\begin{align*}
\frac{\partial^{k}\Phi}{\partial h^{k}}(0,u)
=\left(L_{1}[\Phi^{\pm}]\right)^{k}I(u)=\left(L_{1}[\phi]\right)^{k}I(u)
\end{align*}
and the theorem follows from lemma \ref{le: flujo}.
\end{proof}

\subsubsection{Symmetric case}
If $\phi_{0},\phi_{1}$ were reversible flows, then
it would hold $\Phi^{-}(h)\circ\Phi^{+}(-h)=I$ and using lemma \ref{le: phi psi} we would obtain that $M_{k}$, defined below by \eqref{eq: Mk=0}, is identically zero.
We get the same result for irreversible flows:
\begin{lemma}
\label{le: lemaprevio}
Let $M_{k}:\mathcal{D}_{0}\to\mathcal{D}_{k}$ be the operator given by
\begin{align}
\label{eq: Mk=0}
M_{k}=\sum_{j=0}^{k} (-1)^{j}\binom{k}{j} L_j[\Phi^{+}]L_{k-j}[\Phi^{-}],
\end{align}
then $M_{k}=0$.
\end{lemma}

\begin{proof}
Using lemma \ref{le: phi psi} for $\Phi^{\pm}$ and lemma \ref{le: flujo},
\begin{align*}
M_{k}=&\,
\sum_{j=0}^k \sum_{i=0}^j \sum_{l=0}^{k-j}
(-1)^{j}\binom{k}{j} \binom{j}{i} \binom{k-j}{l}
L_{1}[\phi_{0}]^{j-i}L_{1}[\phi_{1}]^{k+i-j-l}L_{1}[\phi_{0}]^{l}.
\end{align*}
Interchanging the order of summation, considering $n=j-i$ and using the identity
\[
 \binom{k}{n+i} \binom{n+i}{i} \binom{k-n-i}{l}=\binom{k-n-l}{i} \frac{k!}{n!l!(k-n-l)!},
\]
we can write $M_{k}$ as
\begin{align*}
M_{k}=&\,
\sum_{n=0}^k (-1)^n \sum_{l=0}^{k-n-1}
\left(\sum_{i=0}^{k-n-l} (-1)^{i} \binom{k-n-l}{i}\right) \frac{k!}{n!l!(k-n-l)!}\times\\
&\times L_{1}[\phi_0]^{n} L_{1}[\phi_1]^{k-n-l}L_{1}[\phi_0]^{l}
+\sum_{n=0}^k (-1)^n \binom{k}{k-n} L_{1}[\phi_0]^{k}.
\end{align*}
Since $\sum_{i=0}^{k-n-l} (-1)^i \binom{k-n-l}{i}=0$, we have the result.
\end{proof}
\begin{proposition}
\label{le: dpsim}
For $m\geq 1$ it holds that
\[
L_{k}[\Phi_m^{-}]=(-1)^{k}\sum_{r=1}^{k}C_{m,r}\sum_{\beta\in I_{r,k}}\frac{k!}{\beta!}
L_{\beta_{1}}[\Phi^{+}]\ldots L_{\beta_{r}}[\Phi^{+}],
\]
where $C_{m,r}=(-1)^{r}\binom{m+r-1}{r}$.
\end{proposition}

\begin{proof}
We proceed by induction in $m$ and in $k$: for $m=1$, eliminating $L_{k}[\Phi^{-}]$ from \eqref{eq: Mk=0} we have
\begin{align*}
L_{k}[\Phi^{-}]=-\sum_{j=1}^{k} (-1)^{j}\binom{k}{j} L_j[\Phi^{+}]L_{k-j}[\Phi^{-}],
\end{align*}
by inductive hypothesis for $k-j<k$ and using remark \ref{obs: Irk} we obtain the case $m=1$.
Applying lemma \ref{le: phi psi} to $\Phi_{m+1}^{-}=\Phi^{-}\circ\Phi_{m}^{-}$ 
and using $C_{m+1,r}=\sum_{s=0}^{r}C_{m,s}C_{1,r-s}$, we have the result.
\end{proof}

\begin{proposition}
\label{pr: Lphi}
If $\gamma$ satisfies conditions \eqref{eq: cond sim},
then the method $\Phi$ defined by \eqref{eq: metodo afin simetrico}
satisfies $L_{k}[\Phi]I=(L_{1}[\Phi^{+}])^{k}I$ for $k=0,\ldots,2n$.
\end{proposition}
\begin{proof}
Applying proposition \ref{pr: L1Lr} to $\Phi^{+}$, using proposition \ref{le: dpsim} and
lemma \ref{le: combinatorio simetrico} the result may be concluded.
\end{proof}

\begin{proof}{\it (theorem \ref{th: consistencia simetrico})}
From proposition \ref{pr: Lphi} we have
\begin{align*}
\frac{\partial^{k}\Phi}{\partial h^{k}}(0,u)
=(L_{1}[\Phi^{+}]^{k}I)(u)=(L_{1}[\phi]^{k}I)(u),
\end{align*}
and the theorem follows from lemma \ref{le: flujo}.
\end{proof}

\subsection{Stability}
\label{ssec 32}

Assume that $A_0$ and $A_1$ are Lipschitz continuous maps. Using Duhamel integral and Gronwall inequality one can deduce that the associated flows $\phi_0$ and $\phi_1$ and the affine method $\Phi$ are uniformly stable.
Except for ordinary differential equations, this is not the case.
However, if $A=A_{0}+A_{1}$, where $A_{0}$ is the infinitesimal generator of quasicontraction
semi-group and $A_{1}$ is a locally Lipschitz continuous map, we show that the affine methods are stable.
\begin{proposition}
\label{pr: Phi estable}
Let $\phi_{0}$ be a quasicontraction semi-group, that is
\[
\|\phi_{0}(h,u)\|_{\mathsf{H}}\leq \mathrm{e}^{\kappa_{0} h}\|u\|_{\mathsf{H}}
\]
and $\phi_{1}$ a uniformly stable map, then the method $\Phi$ given by \eqref{eq: metodos afines} is a stable map.
\end{proposition}

\begin{proof}
We give the proof only for the symmetric case \eqref{eq: metodo afin simetrico}. 
Using $\phi_{0}=2\sum_{m=1}^{s}\gamma_{m}\phi_{0}$, we see that
$\Phi=\phi_{0}+\sum_{m=1}^{s}\gamma_{m}(\psi_{m}^{+}+\psi_{m}^{-})$,
where $\psi_{m}^{\pm}(h)=\phi_{m}^{\pm}(h/m)-\phi_{0}(h)$. Thus, we have
\begin{align*}
\Lambda(\Phi,h)-1\leq\Lambda(\phi_{0},h)-1
+\sum_{m=1}^{s}|\gamma_{m}|(\Lambda(\psi_{m}^{+},h)+\Lambda(\psi_{m}^{-},h)).
\end{align*}
We use an inductive argument to show that
\begin{align}
\label{eq: Lambda/h}
\limsup_{h\downarrow 0}h^{-1}\Lambda(\psi_{m}^{\pm},h)\leq\mu(\phi_{1}).
\end{align}
For $m=1$, since $\psi_{1}^{+}(h)=(\phi_{1}(h)-I)\circ\phi_{0}(h),
\psi_{1}^{-}(h)=\phi_{0}(h)\circ(\phi_{1}(h)-I),$
we obtain that
\[\Lambda(\psi_{1}^{\pm},h)\leq \Lambda(\phi_{1}-I,h)\Lambda(\phi_{0},h)\] 
and using that
$\lim_{h\downarrow 0}\Lambda(\phi_{0},h)=1$, it yields
\[
\limsup_{h\downarrow 0}h^{-1}\Lambda(\psi_{1}^{\pm},h)\leq\mu(\phi_{1}).
\]
For $m>1$, it holds that
\begin{align*}
\psi_{m}^{\pm}(h)=&\,\psi_{1}^{\pm}(h/m)\circ\phi_{m-1}^{\pm}(h/m)
+\phi_{0}(h/m)\circ\psi_{m-1}^{\pm}((m-1)h/m),
\end{align*}
hence
\begin{align*}
\Lambda(\psi_{m}^{\pm},h)\leq 
&\,\Lambda(\psi_{1}^{\pm},h/m)\Lambda(\phi_{m-1}^{\pm},h/m)
+\Lambda(\phi_{0},h/m)\Lambda(\psi_{m-1}^{\pm},(m-1)h/m).
\end{align*}
By inductive hypothesis and since $\lim\limits_{h\to 0}\Lambda(\phi_{m-1}^{\pm},h)=1$ for all $m$,
we get \eqref{eq: Lambda/h} and therefore
$\kappa(\Phi)\leq \kappa_{0}+2\sum_{m=1}^{s}|\gamma_{m}|\mu(\phi_{1})$.
\end{proof}

\subsection{Convergence results}

The proof of convergence falls naturally from consistency and stability in the usual way.
For the sake of completeness we will give a general result in this regard.

\begin{theorem} \label{th: general_converg}
Let $\Phi \in C([0,T]\times H,H)$ and $u\in C([0,T],H)$ such that:
\begin{enumerate}
 \item Given $R>0$, there exists $\kappa>0$ such that $\|\Phi(t,u)-\Phi(t,v)\|_H\leq \mathrm{e}^{\kappa t} \|u-v\|_H$ 
for all $t\in [0,T]$ and $u,v \in B_R(0)$.
\item There exists a constant $C>0$ such that 
\begin{equation} \label{eq: resto taylor}
\|u(t+h)-\Phi(h,u(t))\|\leq C h^{q+1}.
\end{equation}
\end{enumerate}
Given $u_{0}\in\mathsf{H}$, there exists $\delta$ such that if 
$U_{0}\in\mathsf{H}$ satisfies $\|u_{0}-U_{0}\|_{\mathsf{H}}<\delta$ and $0<h<T$, then the
sequences $U_{n}=\Phi(h,U_{n-1})$ and $u_{n}=u(nh)$ are defined for $n \leq [T/h]$ and satisfies
\begin{align*}
\|u_{n}-U_{n}\|_{\mathsf{H}}\leq \mathrm{e}^{\kappa n h}\|u_{0}-U_{0}\|_{\mathsf{H}}
+C\frac{\mathrm{e}^{\kappa n h}-1}{\kappa}h^{q}.
\end{align*}
\end{theorem}

\begin{proof}
The proof is by induction on $n$.
Let $R=2\max\limits_{t\in [0,T]}\|u(t)\|_{\mathsf{H}}$, and $\kappa$ given by 1),
taking $\delta>0$ 
\[
 \mathrm{e}^{\kappa T}\delta+C\frac{\mathrm{e}^{\kappa T}-1}{\kappa}h^{q}<R/2,
\]
using inductive hypothesis, we obtain
\begin{align*} 
\|U_{n-1}\|_{\mathsf{H}}\leq &\,\|u_{n-1}\|_{\mathsf{H}}+
\|U_{n-1}-u_{n-1}\|_{\mathsf{H}}\\
\leq &\,R/2+\mathrm{e}^{\kappa(n-1)h}\|u_{0}-U_{0}\|_{\mathsf{H}}
+C\frac{\mathrm{e}^{\kappa(n-1) h}-1}{\kappa}h^{q}\\
\leq &\,R/2+\mathrm{e}^{\kappa T}\delta+C\frac{\mathrm{e}^{\kappa T}-1}{\kappa}h^{q}<R.
\end{align*}
From (1) we get that $\|\Phi(h,u_{n-1})-\Phi(h,U_{n-1})\|_{\mathsf{H}}\leq
\mathrm{e}^{\kappa h}\|u_{n-1}-U_{n-1}\|_{\mathsf{H}}$ and therefore using (2) we obtain
\begin{equation}
\|u_{n}-U_{n}\|_{\mathsf{H}}\leq \mathrm{e}^{\kappa h}\|u_{n-1}-U_{n-1}\|_{\mathsf{H}}
+C h^{q+1}.
\end{equation}
Using $\mathrm{e}^{\kappa h}\geq 1+\kappa\, h$, the proof is complete.
\end{proof}

The result of convergence concerning problem \eqref{eq: ut=Au} will be deduced as a corollary of the latter 
theorem \eqref{th: general_converg}, for which we will need some assumptions that are not particularly restrictive in our context. We will assume: 

\begin{enumerate}
 \item \label{as: 1} $\phi$, $\phi_{0}$, $\phi_{1}$ are compatible with $\{\mathsf{H}_{k}\}_{k\geq 0}$,
a sequence of Hilbert spaces with
$\mathsf{H}_{0}=\mathsf{H}$ and $\mathsf{H}_{k+1}\hookrightarrow\mathsf{H}_{k}$.
\item \label{as: 2}  Given $R>0$, there exists $h^*>0$ such that for any $k\ge 0$, if $u\in B_{\mathsf{H}}(0,R)\cap\mathsf{H}_{k}$, then
$\phi(t,u)$, $\phi_{0}(t,u)$ and $\phi_{1}(t,u)$ are defined on $[0,h^*]$.
\item \label{as: 3} The maps $\phi_{0}$, $\phi_{1}$ satisfy the hypothesis of proposition \ref{pr: Phi estable} on 
$B_{\mathsf{H}}(0,R)$.
\end{enumerate}

\begin{remark} \label{rm: tiempo exist}
Note that \eqref{as: 2} implies, by decreasing $h^*$ if necessary, $\phi_{m}^{\pm}(t,u_{0})$ and $\Phi(t,u_{0})$
are defined on $[0,h^*]$.
\end{remark}

\begin{remark}
\label{rm: 311}
These conditions may seem too restrictive, nevertheless they are satisfied in many evolution problems.
As an example, we consider the NLS equation with $\mathsf{H}=\mathrm{H}^{\sigma}(\mathbb{R}^{d})$ the
Sobolev spaces consisting of the $\sigma$ times derivable functions and
$\mathsf{H}_{k}=\mathrm{H}^{\sigma+2k}(\mathbb{R}^{d})$. 
Clearly, the unitary group generated by $\mathrm{i}\Delta $ is compatible with  $\{\mathsf{H}_{k}\}_{k\geq 0}$.
It is known that if $\sigma>d/2$, the spaces $\mathsf{H}_{k}$ 
are Banach algebras with the punctual product of functions, therefore any application as $A_{1}(u)=P(u,u^{*})$, 
where $P$ is a polynomial such that $P(0,0)=0$, turns out to be locally Lipschitz in $\mathsf{H}_{k}$, implying the 
existence of the flow $\phi_{1}$.
Being $A_{1}$ a polynomial application, is infinitely derivable and its derivatives are locally Lipschitz, proving that the 
flow $\phi_{1}$ is compatible with $\{\mathsf{H}_{k}\}_{k\geq 0}$.
From the following estimate
\[
\|A_{1}(u)\|_{\mathsf{H}_{k}}\leq C(\|u\|_{\mathsf{H}})\|u\|_{\mathsf{H}_{k}},
\]
we deduce that the times of existence of the solutions do not depend on $k$.
We refer to \cite{Cazenave2003} for the proof of the mentioned properties of the flow $\phi$ associated to the NLS initial value problem.
\end{remark}


\begin{corollary}
\label{cr: convergence}
Let $\phi_0,\phi_1$ be the associated flows of the partial problems \eqref{eq: ut=A0u}, \eqref{eq: ut=A1u} 
and $\phi$ the flow of \eqref{eq: ut=Au} satisfying assumptions \eqref{as: 1},\eqref{as: 2} and \eqref{as: 3}.
Let $\Phi$ be defined by \eqref{eq: metodo afin asimetrico} or \eqref{eq: metodo afin simetrico} with 
$\gamma=(\gamma_1,\ldots,\gamma_s)$ 
satisfying \eqref{eq: cond asim} or \eqref{eq: cond sim} respectively. Then, given $u_0\in H_{q+1}$ and 
$u(t)=\phi(t,u_0)$
the maximal solution of \eqref{eq: ut=Au} defined on $[0,T_*)$, for any $T\in (0,T_*)$ there exist $h_*,\delta,
\kappa,C$
such that if $U_0\in H_{q+1}$ satisfies $\|u_0-U_0\|<\delta$ and $0<h<h_*$, then the sequence 
$U_n=\Phi(h,U_{n-1})$ 
is defined for $n\leq [T/h]$ and satisfies
\begin{align*}
\|\phi(nh,u_0)-U_{n}\|_{\mathsf{H}}\leq \mathrm{e}^{\kappa n h}\|u_{0}-U_{0}\|_{\mathsf{H}}
+C\frac{\mathrm{e}^{\kappa n h}-1}{\kappa}h^{q}.
\end{align*}
\end{corollary}

\begin{proof}
We begin by noting that, as a consequence of remark \ref{rm: tiempo exist}, there exists $h^*>0$ such
that $\phi(t,u)$ and $\Phi(t,u)$ are defined on $[0,h^*]$ for all $u\in B_{\mathsf{H}}(0,R)\cap\mathsf{H}_{q+1}$.
It is enough to prove assumptions (1) and (2) of theorem \ref{th: general_converg}.
Condition (1) is a straightforward consequence of proposition \ref{pr: Phi estable}. 

Being $\phi$ and $\Phi$ compatible with $\{H_k\}_{k\geq 0}$ we have that $\phi,\Phi \in C^{q+1,\infty}([0,h^*]\times H^{q+1},H)$. 
Then, since $u(t+h)=\phi(h,u(t))$, condition (2) is concluded from theorem \ref{th: consistencia asimetrico} (or \ref{th: consistencia simetrico}) and Taylor formula.
\end{proof}

The computation of $\Phi$ requires to solve exactly the partial problems. Besides some simple cases of ordinary 
differential equations, this is not possible. In what follows we will show that we can define integration methods of 
order $q$ using suitable approximations of the flows $\phi_{0}$ and $\phi_{1}$.
Let $\Psi\in C([0,h_{*})\times\mathsf{H},\mathsf{H})$ satisfying
\begin{align}
\label{eq: Psi-Phi}
\|\Psi(h,u)-\Phi(h,u)\|_{\mathsf{H}}\leq \rho
\end{align}
for $u\in B_{\mathsf{H}}(R,0)$. Let $V_{0}=U_{0}$ and $V_{n}=\Psi(h,V_{n-1})$, from the stability of $\Phi$ we 
get that
\[
\|U_{n}-V_{n}\|_{\mathsf{H}}\leq C\frac{\mathrm{e}^{\kappa n h}-1}{\mathrm{e}^{\kappa h}-1}\rho.
\]
Let $\psi_j$ be an approximation of $\phi_j$ such that $\|\phi_{j}(h,u)-\psi_{j}(h,u)\|_{\mathsf{H}}\leq Ch^{q+1}$, for
$j=0,1$. Then the map $\Psi$ defined by \eqref{eq: metodos afines} with $\psi_{j}$ in place of $\phi_{j}$ 
satisfies the condition \eqref{eq: Psi-Phi} with $\rho=C h^{q+1}$ and consequently the method $\Psi$ satisfies
\begin{align}
\label{eq: u-V}
\|u_{n}-V_{n}\|_{\mathsf{H}}\leq \mathrm{e}^{\kappa n h}\|u_{0}-V_{0}\|_{\mathsf{H}}
+M\frac{\mathrm{e}^{\kappa n h}-1}{\kappa}h^{q}.
\end{align}

We consider the following example. Let $\{u_{n}\}_{n\in\mathbb{N}}$ be an orthonormal basis of $\mathsf{H}$ and 
$\phi_{0}(t,u)=\sum_{n\in\mathbb{N}}\mathrm{e}^{\alpha_{n}t}\langle u_{n},u\rangle u_{n}$ 
with $\mathrm{Re}(\alpha_{n})\leq\kappa$.
We define the spaces 
\[
\mathsf{H}_{k}=\{u\in\mathsf{H}:\sum_{n\in\mathbb{N}}|\alpha_{n}|^{2k}|\langle u_{n},u\rangle|^{2}<\infty\},
\]
so that $\phi_{0}$ becomes compatible with $\{\mathsf{H}_{k}\}_{k\geq 0}$ 
and satisfies 
$\|\phi_{0}(t,u)\|_{\mathsf{H}_{k}}\leq \mathrm{e}^{\kappa t}\|u\|_{\mathsf{H}_{k}}$. If we take
$\psi_{0}(t,u)=\sum_{1\leq n\leq N}\mathrm{e}^{\alpha_{n}t}\langle u_{n},u\rangle u_{n}$, we obtain that
\[
\|\phi_{0}(h,u)-\psi_{0}(h,u)\|_{\mathsf{H}}\leq
\mathrm{e}^{\kappa h}\inf_{n> N}|\alpha_{n}|^{-k}\|u\|_{\mathsf{H}_{k}}.
\]
Hence, if $\liminf\limits_{n\to\infty}|\alpha_{n}|=+\infty$, for $h.R>0$, there exists a $N=N(h)$ 
large enough such that $\|\phi_{0}(h,u)-\psi_{0}(h,u)\|_{\mathsf{H}}\leq Ch^{q+1}$ if $u\in\mathsf{H}_{k}$
with $\|u\|_{\mathsf{H}_{k}}\le R$. From theorem 1.2 in \cite{Borgna2015}, we can see that
$U_{n},V_{n}\in B_{\mathsf{H}_{k}}(R,0)$ for $h$ small enough. Therefore, inequality \eqref{eq: u-V} holds.

\section{Numerical examples}
\label{sec 4}
We present several examples which illustrate the performance of the proposed methods.
\subsection{Ordinary differential system}
We begin by considering an elementary example which is simple to deal with the proposed methods,
but it would be more expensive to solve with symplectic methods.
The bidimensional system
\begin{align}
\label{eq: example}
\begin{cases}
\dot{u}_{1}=4u_{2}-\tan(u_{1}),\\
\dot{u}_{2}=-4u_{1}-\tan(u_{2}),
\end{cases}
\end{align}
can be splitted in a linear system and a decoupled system.
The linear flow is a clockwise rotation, orbits are showed in figure \ref{fig: example 1} for concentric circles.
Lines that go through the origin are the orbits of the system $\dot{u}_{j}=-\tan(u_{j})$,
which solution is $u_{j}(t)=\arcsin(\mathrm{e}^{-t}\sin(u_{j,0}))$. Note that solutions are not defined for
$t<\ln|\sin(u_{j,0})|\leq 0$, which implies $h$ should be small for symplectic methods (with negative steps).
For initial data $(1,3/2)$, the solution computed with Runge--Kutta with a very small $h$ is showed in
figure \ref{fig: example 1}, the points are the solution obtained with the symmetric method $\Phi$ of fourth order
with $s=2$, $\gamma_{1}=-1/6$, $\gamma_{2}=2/3$ and $h=0.2$.
It can be seen numerically that for this step, $h=0.2$, the symplectic method proposed in \cite{Neri1987} can not be used.
\begin{figure}[ht]
\begin{center}
\includegraphics[scale=0.5]{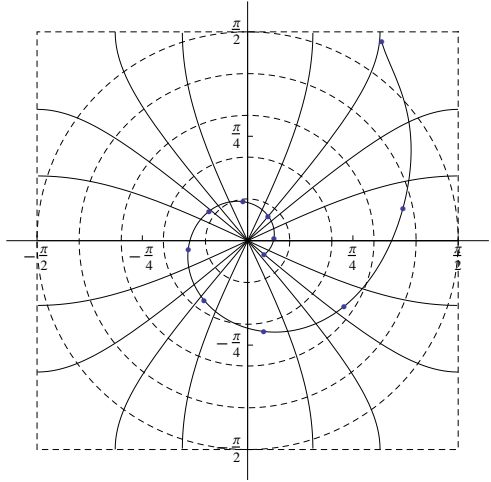}
\caption{Flows $\phi_{0},\phi_{1}$ and solution of \eqref{eq: example} obtained with $\Phi$ of fourth order.}
\label{fig: example 1}
\end{center}
\end{figure}

\subsection{Oscillatory reaction--diffusion system}
In this example, we study the behavior of the methods of a reaction--diffusion system,
as the ones shown in \cite{Kopell1973}.
Since this system is an irreversible problem, symplectic methods with negative steps can not be used.
We consider the system
\begin{align}
\label{eq: PTW}
\begin{split}
\partial_{t}v=&\,\Delta v+(1-r^{2})v-(\omega_{0}-\omega_{1}r^{2})v,\\
\partial_{t}w=&\,\Delta w+(\omega_{0}-\omega_{1}r^{2})v+(1-r^{2})w,
\end{split}
\end{align}
where $r^{2}=v^{2}+w^{2}$.
If $u=v+iw$, equation \eqref{eq: PTW} reads as follows:
\[
\partial_{t}u=\Delta u+(1-|u|^{2})u+\mathrm{i}(\omega_{0}-\omega_{1}|u|^{2})u.
\]
The right hand member can be written as $A_0 u+A_1(u)$, where
$A_{0}u=\Delta u$ and
\[
A_{1}(u)=(1-|u|^{2})u+\mathrm{i}(\omega_{0}-\omega_{1}|u|^{2})u.
\]
The flow $\phi_{1}$ is given by
\[
\phi_{1}(h,u)=u\mathrm{e}^{h}(1+(\mathrm{e}^{2h}-1)|u|^{2})^{-1/2}
\mathrm{e}^{\mathrm{i}(\omega_{0}h-\omega_{1}/2\ln(1+(\mathrm{e}^{2h}-1)|u|^{2}))}.
\]
We will restrict our discussion to $L$--periodic solutions,
flow $\phi_{0}$ can be computed approximately by using discrete Fourier transform (DFT).
Let $\eta$ be an odd integer, $\eta=2l+1$ with $l\in \mathbb{N}$, consider
\[
(I_{\eta}u)(x)=\sum\limits_{\nu =-{l}}^{{l}}\hat{U}_{\nu }\mathrm{e}^{\mathrm{i} a \nu  x},
\]
where $a=2\pi/L$ and 
$\hat{U}_{\nu }$ is the DFT coefficient given by
\[
\hat{U}_{\nu }=\frac{1}{\eta}\sum_{r=0}^{\eta-1}U_{r}\mathrm{e}^{-\mathrm{i} 2\pi r\nu/\eta}=\frac{1}{\eta}
\sum_{r=0}^{\eta-1}u(L r/\eta)\mathrm{e}^{-\mathrm{i}2\pi r\nu/\eta}.
\]
Since $\mathrm{e}^{-\mathrm{i}2\pi r\nu/\eta}=\mathrm{e}^{-\mathrm{i}2\pi r(\nu \pm \eta)/\eta}$, it holds that 
$\hat{U}_{\nu }=\hat{U}_{\nu \pm \eta}$.
From lemma 2.2. in \cite{Tadmor1986}, for $u\in \mathrm{H}^{\sigma}(\mathbb{T})$
with $\sigma>1/2$ we have that
\[
\|u-I_{\eta}u\|_{L^{2}(\mathbb{T})}\leq
C_{L,\sigma} \eta^{-\sigma}\|u\|_{\mathrm{H}^{\sigma}(\mathbb{T})}.
\]
Thus we can derive the following proposition.
\begin{proposition}
\label{pr: Tadmor}
Let $\psi_{0}(h)=\phi_{0}(h)I_{\eta}$, then for $u\in \mathrm{H}^{\sigma}(\mathbb{T})$ with $\sigma>1/2$ it holds that
\begin{align*}
\|\psi_{0}(h)u-\phi_{0}(h)u\|_{L^{2}(\mathbb{T})}\leq C_{L,\sigma} 
\eta^{-\sigma}\|u\|_{\mathrm{H}^{\sigma}(\mathbb{T})}.
\end{align*}
\end{proposition}
From the definition of $\psi_{0}(h)$ and using that $\hat{U}_{\nu }=\hat{U}_{\nu \pm \eta}$, we get
\begin{align*}
(\psi_{0}(h)u)(L r/\eta)=&\,\sum_{\nu =-l}^{l}\hat{U}_{\nu }\mathrm{e}^{-a^{2}\nu ^{2}t}
\mathrm{e}^{\mathrm{i}2\pi r\nu/\eta}
=\sum_{\nu =l+1}^{\eta-1}\hat{U}_{\nu }\mathrm{e}^{-a^{2}(\eta-\nu )^{2}t}\mathrm{e}^{\mathrm{i}2\pi r\nu/\eta}\\
&\,+\sum_{\nu =0}^{l}\hat{U}_{\nu }\mathrm{e}^{-a^{2}\nu ^{2}t}\mathrm{e}^{\mathrm{i}2\pi r\nu/\eta}
=\sum_{\nu =0}^{\eta-1}\hat{U}_{\nu }\mathrm{e}^{-a^{2}\lambda_{\nu }t}\mathrm{e}^{\mathrm{i}2\pi r\nu/\eta},
\end{align*}
where $\lambda_{\nu }=\eta^{2}g(\nu /\eta)$ for $0\leq \nu \leq \eta-1$ and $g(\xi)=\xi^{2}-2(\xi-1/2)_{+}$.

In \cite{Kopell1973} the stability of the planar waves
\begin{align*}
v(x,t)=&\,r^{*}\cos(\theta_{0}\pm a x+(\omega_{0}-\omega_{1}r^{*2})t),\\
w(x,t)=&\,r^{*}\sin(\theta_{0}\pm a x+(\omega_{0}-\omega_{1}r^{*2})t),
\end{align*}
is proven, if $L>2\pi(3+2\omega_{1}^{2})^{1/2}$, where $r^{*}=L^{-1}(L^{2}-4\pi^{2})^{1/2}$ 
and $\theta_{0}$ is an arbitrary constant (see also \cite{Sherratt2003}).
Taking $L=4\pi$, $\omega_{0}=1$, $\omega_{1}=1/2$ and $u_{0}=r^{*}\mathrm{e}^{iax}$,
we compare methods given by \eqref{eq: metodo afin simetrico} of order $q=4,6,8$ with $\eta=63$. 
A similar analysis to that in remark \ref{rm: 311} for the quasicontraction semi-group generated by $\Delta$ 
shows that the hypothesis of corollary \ref{cr: convergence} are satisfied.
The fourth order method used is the same as the previous example, 
for the sixth order method we take $s=3$, $\gamma_{1}=1/48$, $\gamma_{2}=-8/15$ and $\gamma_{3}=81/80$,
for the eighth order method we take $s=4$, $\gamma_{1}=-1/720$, $\gamma_{2}=8/45$, 
$\gamma_{3}=-729/560$ and $\gamma_{4}=512/315$.
In figure \ref{fig: Errores_globales_reaccion_difusion} global errors for $T=10$ are shown.
We note that the slopes coincide with the expected order up
to the point where the rounding error dominates the total error.
\begin{figure}[ht]
\begin{center}
\includegraphics[scale=0.65]{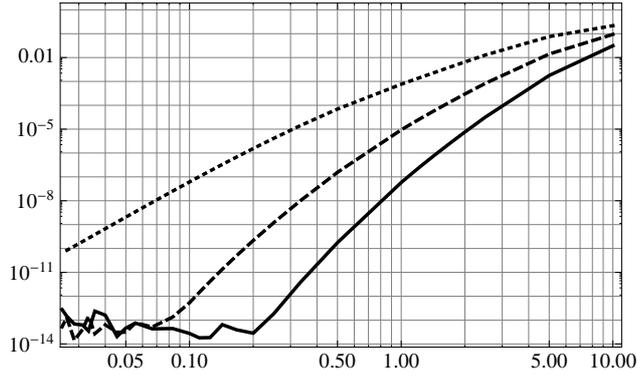} 
\caption{Global Error of $\Phi$ vs. $h$ for $q=4,6,8$}
\label{fig: Errores_globales_reaccion_difusion}
\end{center}
\end{figure}
In order to show the stability of the planar waves, we consider the initial data
$\tilde{u}_{0}(x)=0.8u_{0}(x)+0.1+2.5\mathrm{e}^{\mathrm{i}2ax}-0.8\mathrm{i}\mathrm{e}^{\mathrm{i}3ax}$. In 
figure \ref{fig: evolucion} we can see the evolution of 
the fourth order method $\Phi(t,\tilde{u}_{0})$ for $t\in [0,50]$,
calculated with $\eta=63$ and $h=0.1$ and $\phi(t,u_{0})$ is showed in dashed line.

\begin{figure}[ht]
\begin{center}
\includegraphics[height=0.275\hsize]{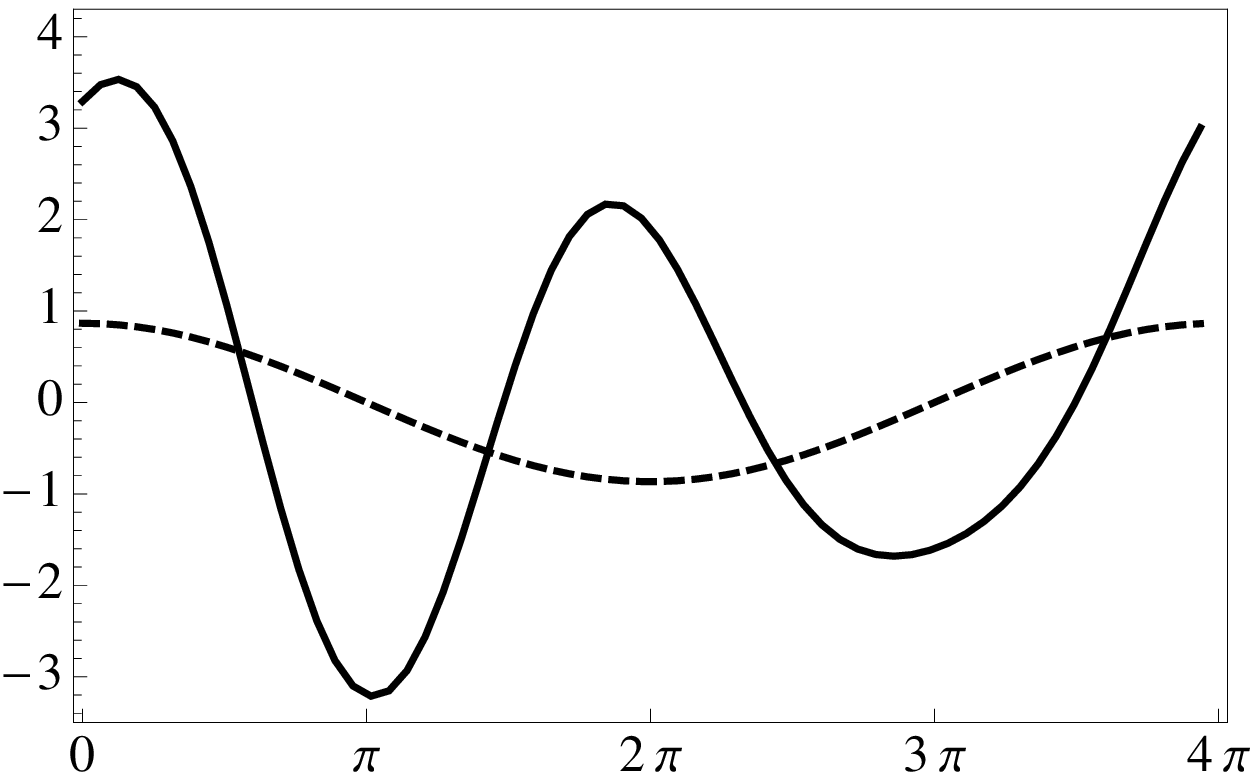} 
\qquad
\includegraphics[height=0.275\hsize]{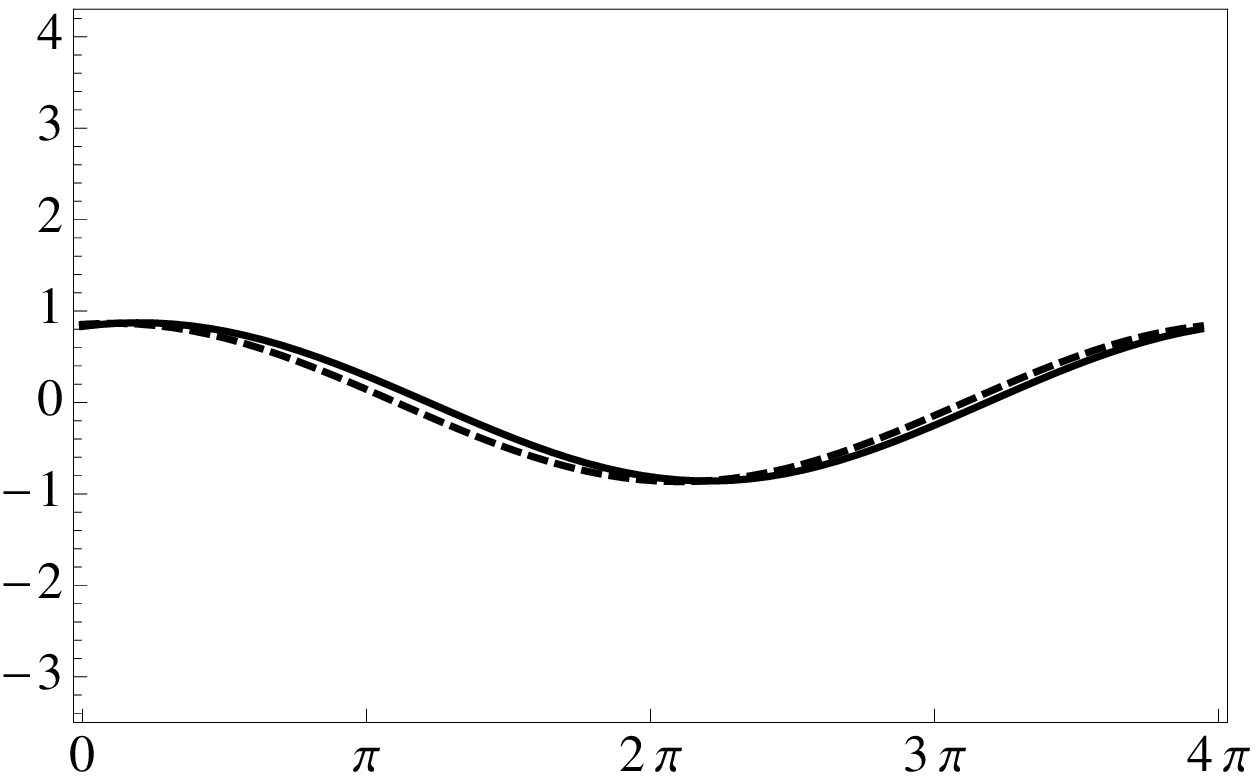} 
\caption{\label{fig: evolucion}$\mathrm{Re}(\Phi(t,\tilde{u}_{0}))$ for $t=0$ (left) and $t=50$ (right)}
\end{center}
\end{figure}

\subsection{Regularized Schr\"odinger--Poisson equation}
In this example, we study the $2\pi$--periodic solutions of the regularized Schr\"odinger--Poisson equation
\begin{align}
\label{eq: Schroedinger cubica}
\begin{cases}
\partial_{t}u=\mathrm{i}\partial_{x}^{2} u-(-\partial_{x}^{2})^{\beta}u+\mathrm{i}|u|^{2}u
+\mathrm{i}(G\ast|u|^{2})u,\\
u(0)=u_{0},
\end{cases}
\end{align}
where $0<\beta<1$ and $G$ is a real kernel. 
Similar equations are considered in \cite{Aloui2008a}, \cite{Aloui2008} and \cite{Aloui2013}, on bounded domains of $\mathbb{R}^{n}$ as well as on compact manifolds.
In order to apply the methods given by \eqref{eq: metodo afin simetrico}, we consider
the flow $\phi_{0}$ generated by the linear operator $L=\mathrm{i}\partial_{x}^{2}-(-\partial_{x}^{2})^{\beta}$,
and the flow $\phi_{1}(h,u)=\exp(\mathrm{i}h(|u|^{2}+G\ast|u|^{2}))u$ associated to
$\partial_{t}u=\mathrm{i}(|u|^{2}+G\ast|u|^{2})u$.
If $\rho=|u|^{2}$ and $\rho(x,t)=\sum_{\nu\in\mathbb{Z}}\hat{\rho}_{\nu}(t)\mathrm{e}^{\mathrm{i}\nu x}$, we have
\[
(G\ast|u|^{2})(x,t)=\sum_{\nu\in\mathbb{Z}}\hat{G}_{\nu}\hat{\rho}_{\nu}(t)\mathrm{e}^{\mathrm{i}\nu x}
\]
Both $\phi_{0}$, $\phi_{1}$ can be numerically solved using discrete Fourier transform as in the example above.
Using FFT, the computational cost of each evaluation is $O(\eta\log\eta)$, where $\eta$ is the number of point in the 
spatial discretisation.

In order to analyse the performance of the integrators proposed, we consider the exact solutions
$u(x,t)=r(t)\mathrm{e}^{\mathrm{i}(\nu_{0} x+\theta(t))}$, with $r(t)=r_{0}\mathrm{e}^{-|\nu_{0}|^{2\beta}t}$ and
\[
\theta(t)=-\nu_{0}^{2}\,t+\frac{1}{2}(1+\hat{G}_{0})r_{0}^{2}|\nu_{0}|^{-2\beta}
\left(1-\mathrm{e}^{-2|\nu_{0}|^{2\beta}t}\right)+\theta_{0}.
\]
Note that $u(.,t)$ has only one oscillation mode, and taking $\nu_{0}$ as the momentum of the wave as it is usual,
we can say that $u$ is a monokinetic wave.
As an example, we consider the Poisson kernel given by
\begin{align*}
G(x)=\frac{\sinh(\lambda)}{\cosh(\lambda)-\cos(x)},
\end{align*}
then $\hat{G}_{\nu}=\mathrm{e}^{-\lambda|\nu|}$.
In figure \ref{fig: Errores_globales_Schrodinger_disipativo}, absolute global errors and relative global errors
defined by
\begin{align*}
\mathcal{E}_{\mathrm{abs}}=\max_{0\le n\le [T/h]}\|u_{n}-U_{n}\|_{L^{2}},\quad
\mathcal{E}_{\mathrm{rel}}=\max_{0\le n\le [T/h]}\frac{\|u_{n}-U_{n}\|_{L^{2}}}{\|u_{n}\|_{L^{2}}},
\end{align*}
are shown, with $\beta=1/4$, $T=4$, $\lambda=1$,
initial condition $u_{0}=\mathrm{e}^{\mathrm{i}4x}$ and
methods varying from fourth to fourteenth order.
The number of points in the spatial discretisation is $\eta=31$ and the temporal steps
$h$ ranging from $0.01$ to $2$. Like in the example above the slopes coincide with the expected order up
to the point where the rounding error dominates the total error.

\begin{figure}[ht]
\begin{center}
\includegraphics[height=0.3\hsize]{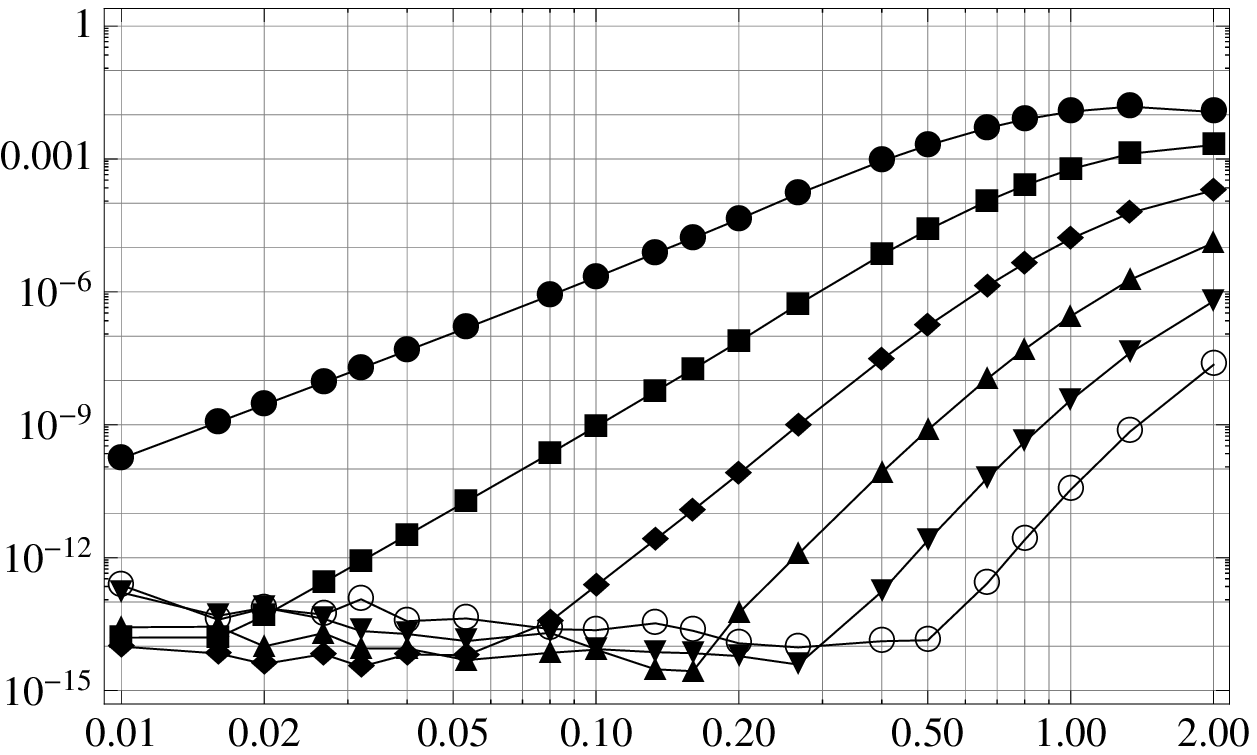} 
\quad
\includegraphics[height=0.3\hsize]{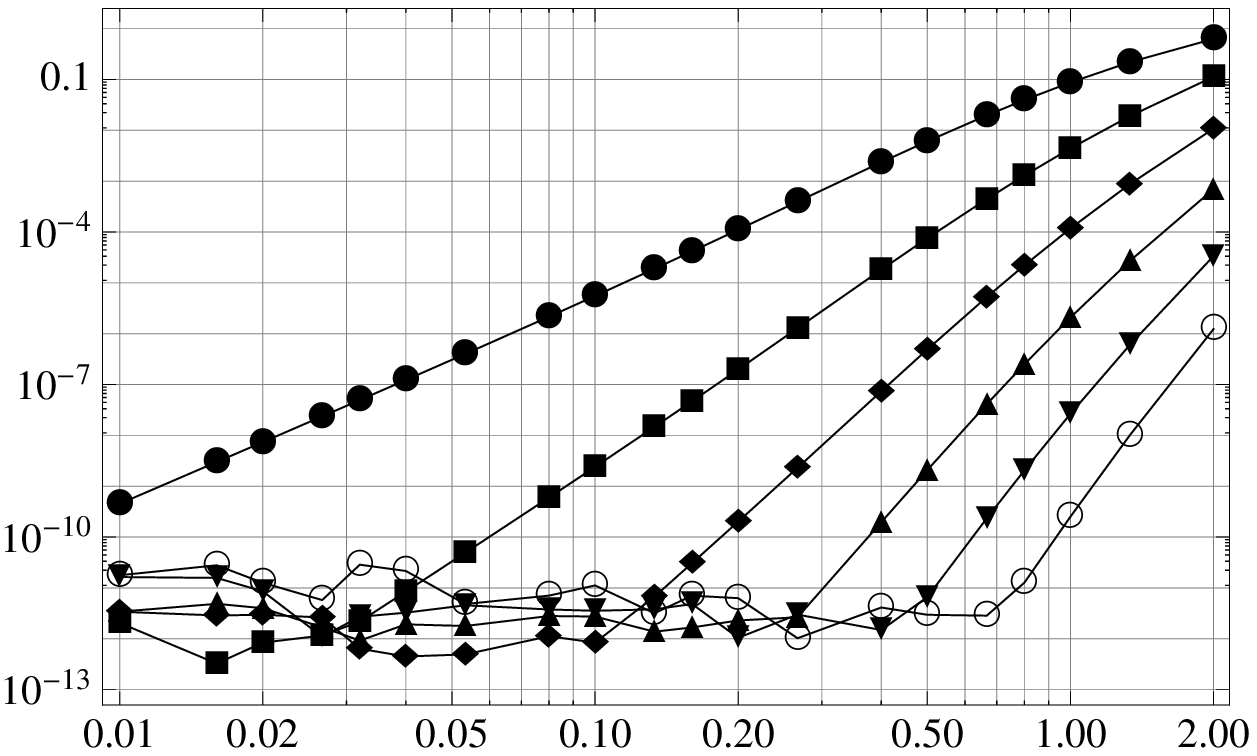} 
\caption{\label{fig: Errores_globales_Schrodinger_disipativo}
Global error vs. $h$ for $q=4,6,\ldots,14$, absolute error (left) and relative error (right)}
\end{center}
\end{figure}

\noindent
For $\nu_{0}=0$, it holds $u(x,t)=r_{0}\mathrm{e}^{\mathrm{i}2|r_{0}|^{2}t+\mathrm{i}\theta_0}$ which are time periodic solutions.
Multiplying \eqref{eq: Schroedinger cubica} by $\bar{u}$ and integrating by parts, we get
\begin{align*}
\frac{d}{dt}\|u\|_{L^{2}}^{2}=-2\|(-\partial_{x}^{2})^{\beta/2}u\|_{L^{2}}^{2}
=-2\mathop{\sum_{\nu\in\mathbb{Z}}}_{\nu\ne 0}|\nu|^{2\beta}|\hat{u}_{\nu}|^{2}
\le -2\|Pu\|_{L^{2}}^{2},
\end{align*}
where $Pu=\sum_{\nu\ne 0}\hat{u}_{\nu}\mathrm{e}^{\mathrm{i}\nu x}$ and
therefore the monokinetic solution with $\nu_{0}=0$ is the only time periodic solution.

It is easy to see that the flow $\phi$ of equation \eqref{eq: Schroedinger cubica} preserves parity, then
for any odd initial data $u_{0}$, $u(t)$ is an odd function and $u(t)=Pu(t)$ for $t>0$.
Therefore, it holds $d\|u\|_{L^{2}}^{2}/dt\le -2\|u\|_{L^{2}}^{2}$
and $\|u\|_{L^{2}}\le \mathrm{e}^{-t}\|u_{0}\|_{L^{2}}$.
We will test the numerical methods by verifying these properties. Consider the odd initial data $u_{0}(x)=\mathrm{e}^{\cos(2x)+\mathrm{i}\pi/6}\sin(5x)$,  in figure \ref{fig: evolucion_Schrodinger_disipativa} we show the numerical solution obtained with the eighth symmetric integrator with $\eta=255$ and $h=0.1$.
Since the higher the frequencies are, the stronger is the damping, $u$ asymptotically behaves like 
$a \mathrm{e}^{-t-it}\sin(x)$.
\begin{figure}[ht!]
\begin{center}
\includegraphics[height=0.175\hsize]{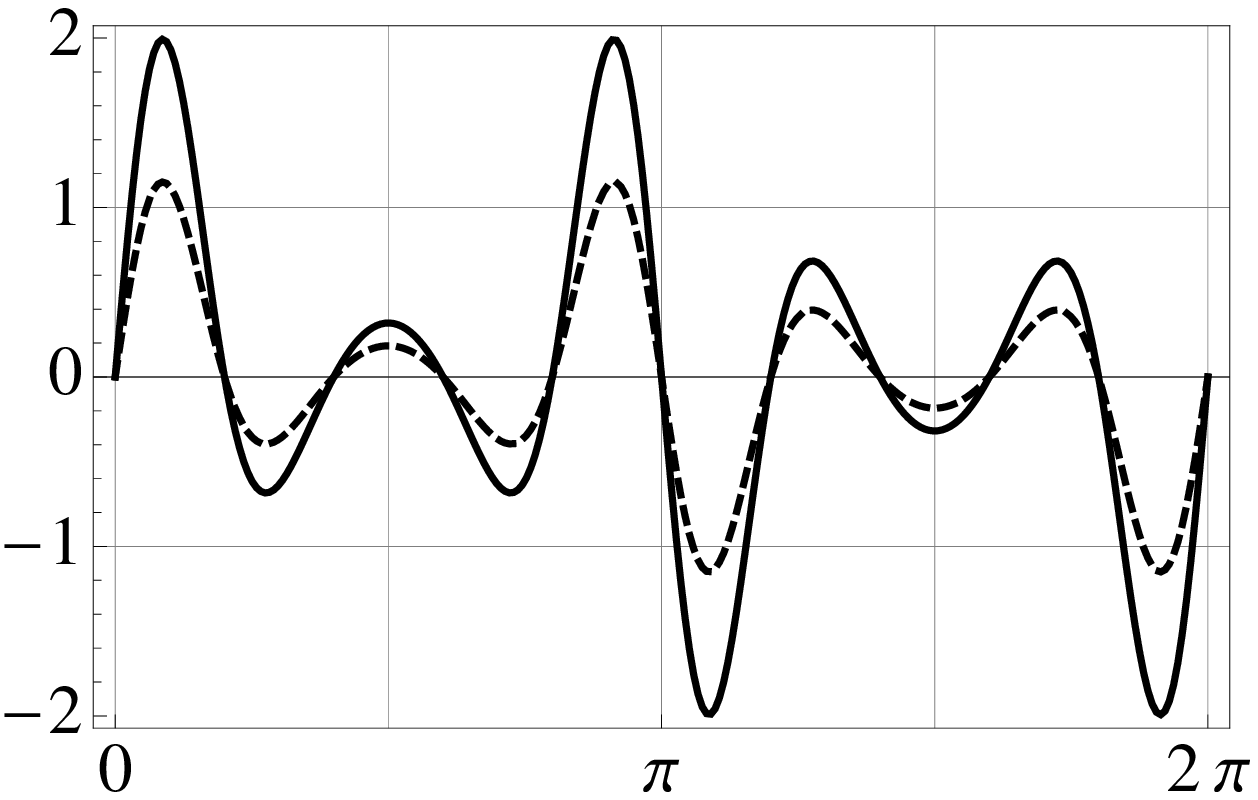} 
\quad
\includegraphics[height=0.175\hsize]{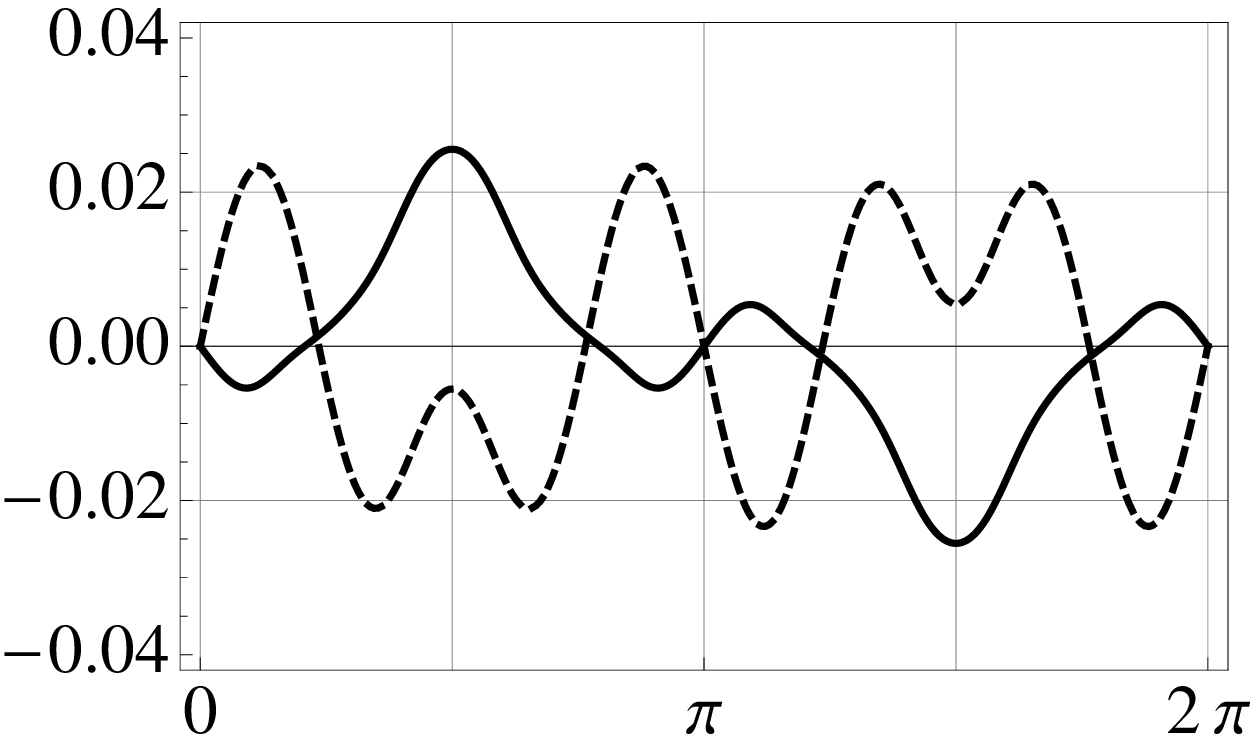} 
\quad
\includegraphics[height=0.175\hsize]{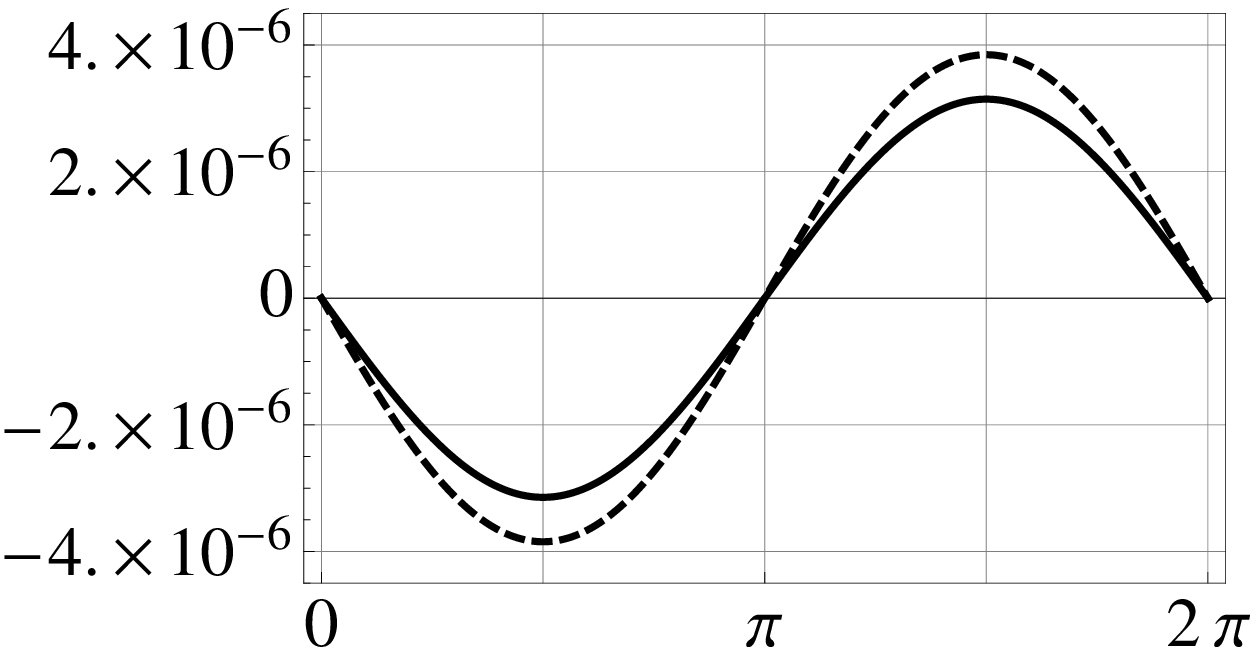} 
\caption{\label{fig: evolucion_Schrodinger_disipativa}
$\mathrm{Re}(\Phi(t,u_{0}))$ and $\mathrm{Im}(\Phi(t,u_{0}))$
for $t=0$ (left), $t=2$ (center) and $t=10$ (right)}
\end{center}
\end{figure}
In figure \ref{fig: impar}, it is shown the evolution of $\|u(.,t)\|_{L^{2}}/\|u_{0}\|_{L^{2}}$ in continuous line,
the function $\mathrm{e}^{-t}$ in dotted line and the asymptotic behaviour in dashed line.
\begin{figure}[ht!]
\centering
\subfigure[Odd solution]{
   \includegraphics[scale =0.5] {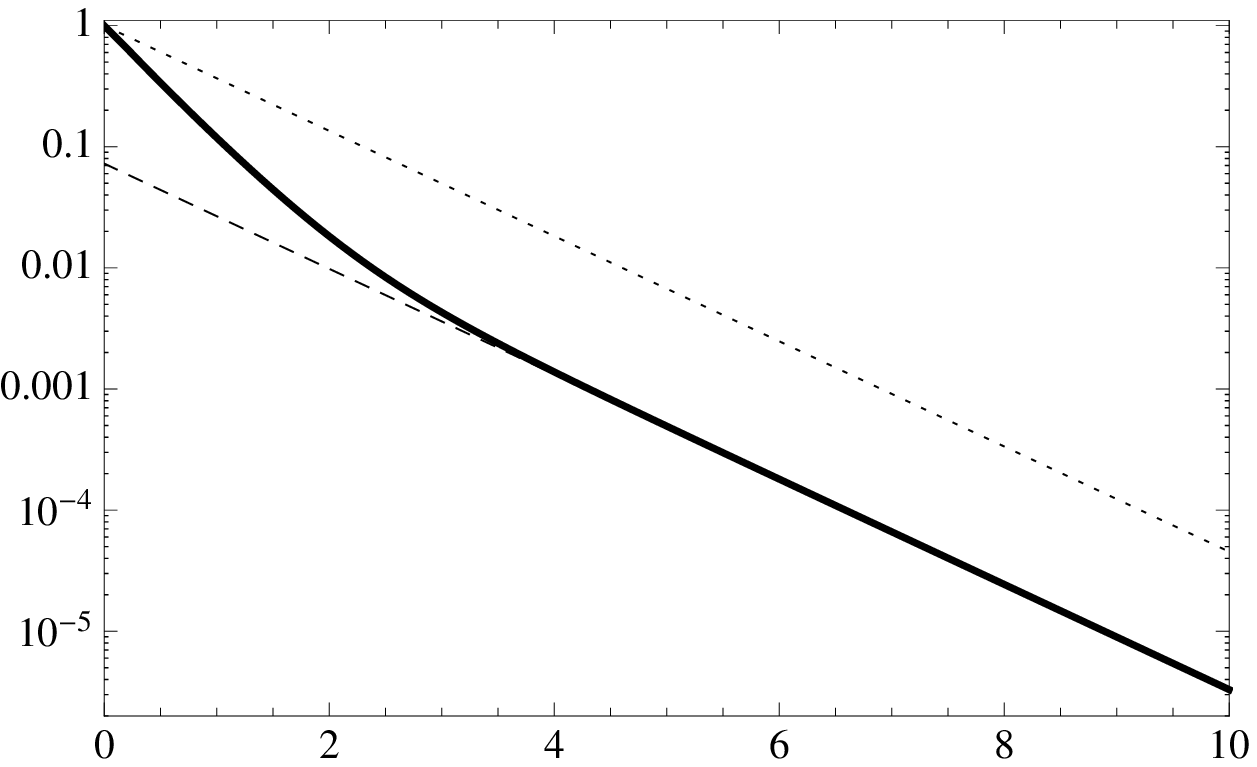}
   \label{fig: impar}}
 \subfigure[Even solution]{
   \includegraphics[scale =0.5] {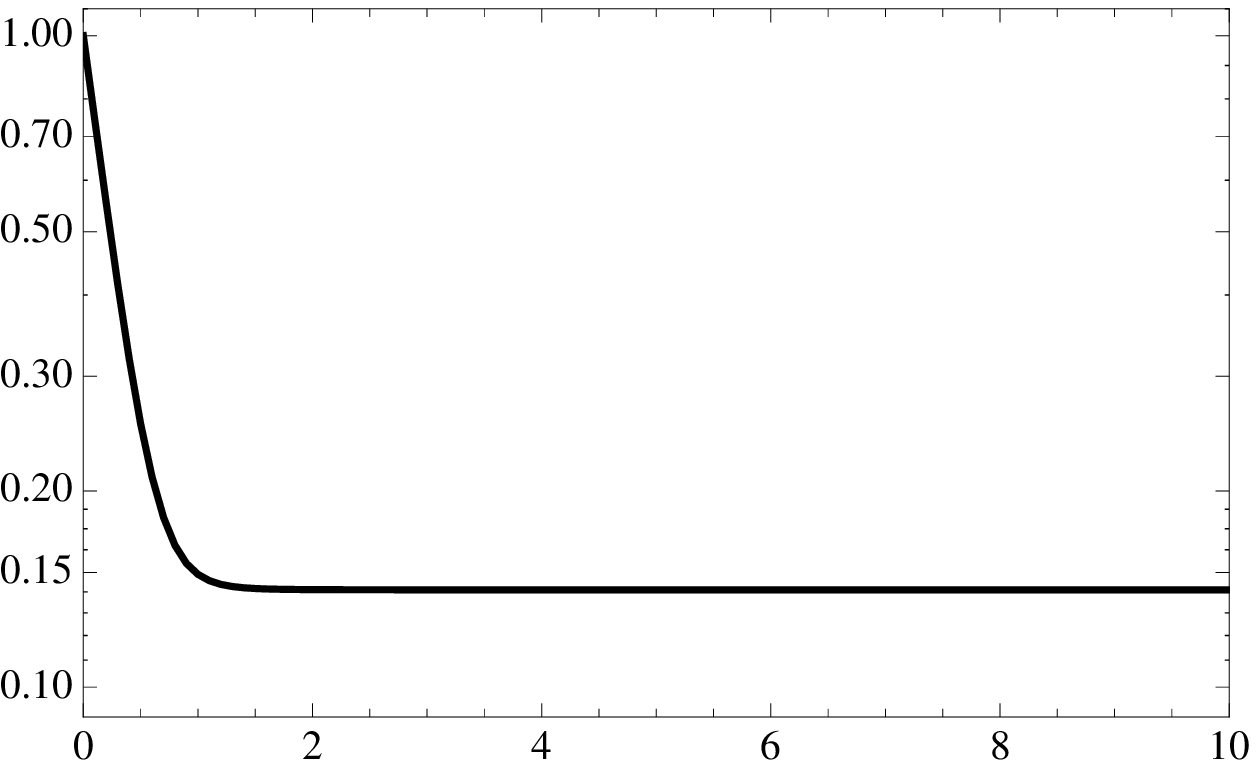}
   \label{fig: par}}
\label{fig: evolucion_norma}
\caption{Evolution of $\|u\|_{L^{2}}/\|u_{0}\|_{L^{2}}$ vs. time}
\end{figure}

We also consider a numerical computation with
$u_{0}(x)=\mathrm{e}^{\cos(2x)+\mathrm{i}\pi/6}(1-1.75\cos^{2}(5x))$ an even initial data.
Using the same integrator as in the odd case, we see that the solution converges to the periodic solution
$u(x,t)\sim a \mathrm{e}^{i2|a|^{2}t}$ as it is seen in figure \ref{fig: evolucion_Schrodinger_disipativa_par}.
In figure \ref{fig: par} it can be observed the fast stabilization of the norm.
This suggests that the periodic solutions are limit cycles of the dynamic given by the equation
\eqref{eq: Schroedinger cubica}.
\begin{figure}[ht!]
\begin{center}
\includegraphics[height=0.175\hsize]{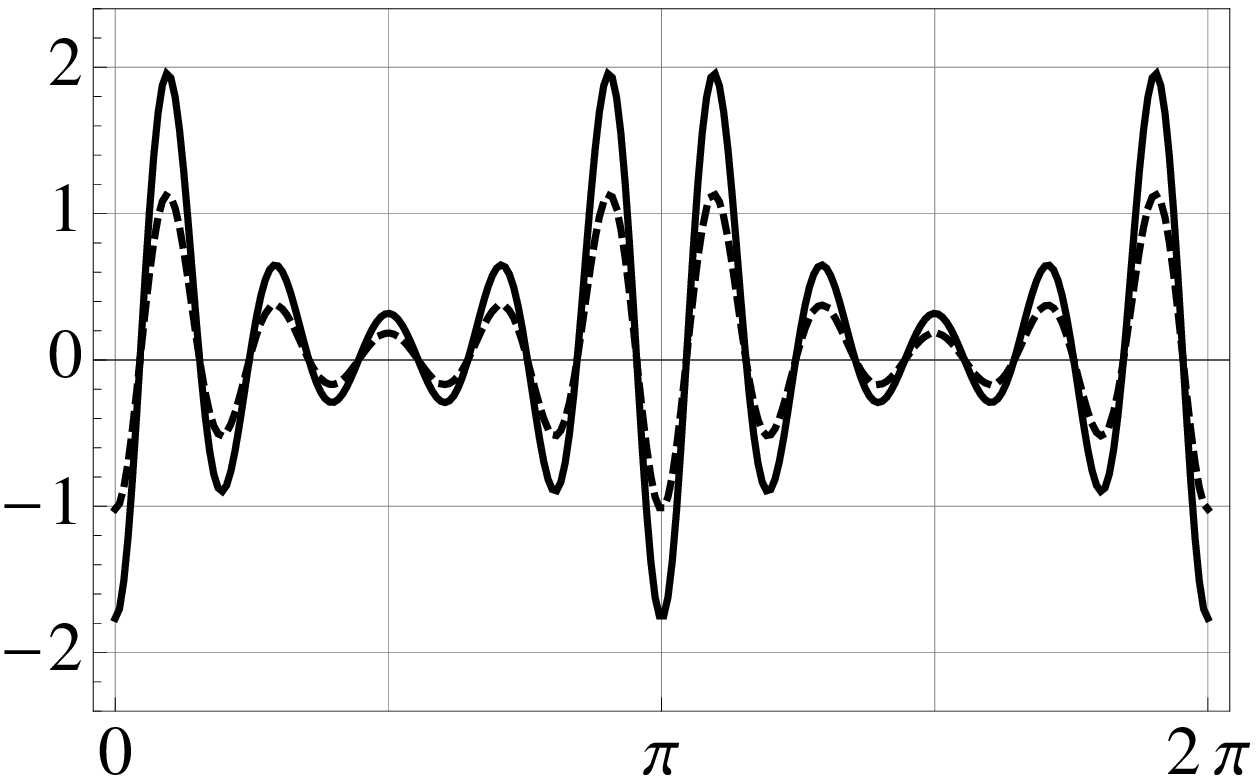} 
\quad
\includegraphics[height=0.175\hsize]{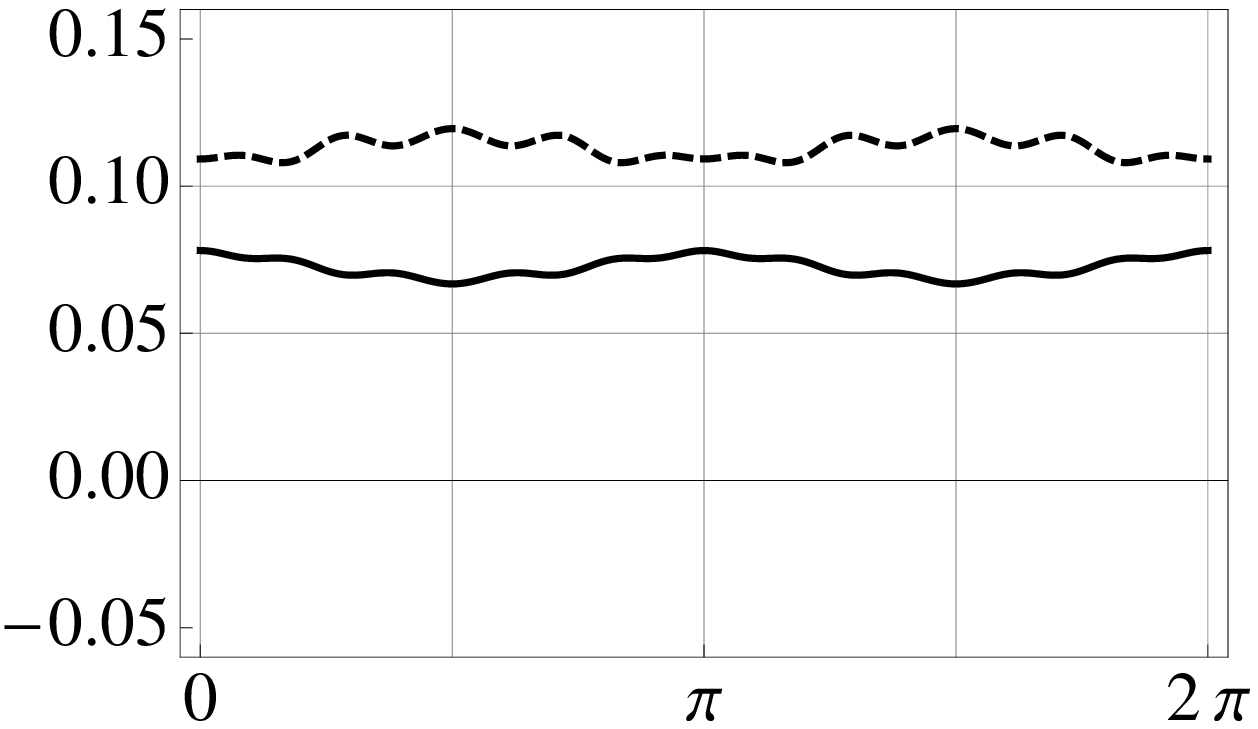} 
\quad
\includegraphics[height=0.175\hsize]{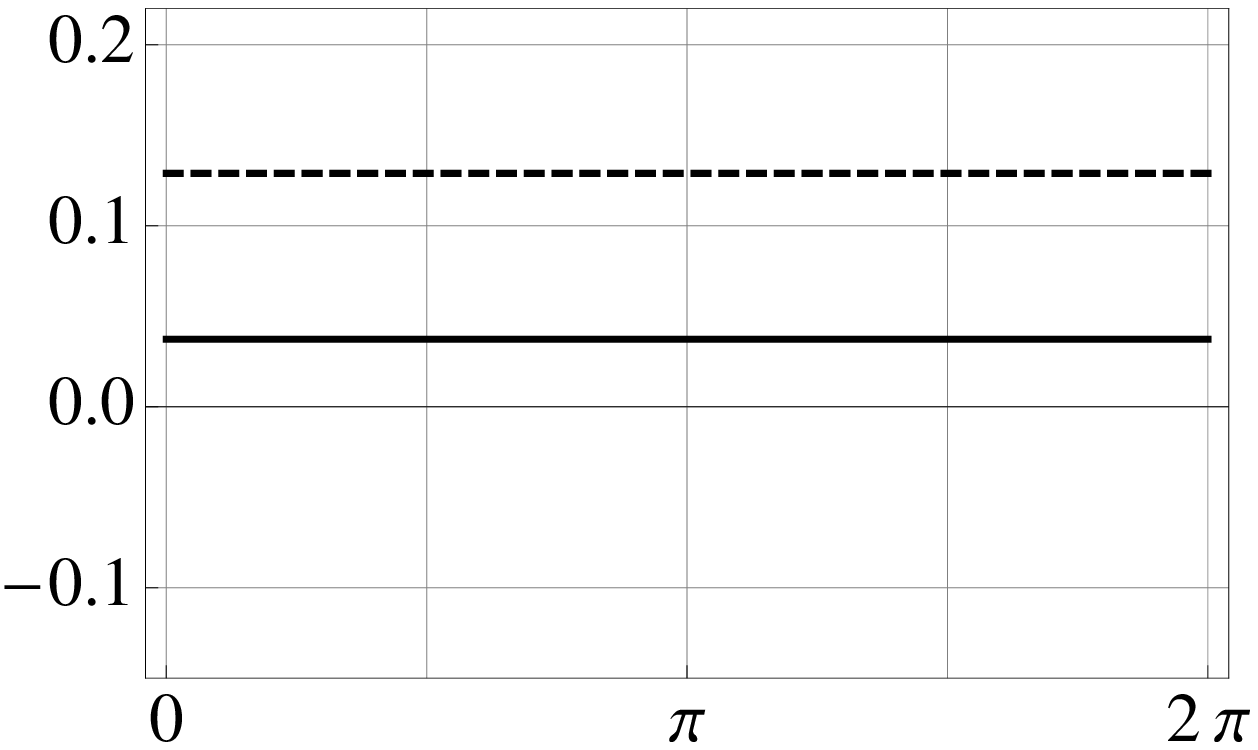} 
\caption{\label{fig: evolucion_Schrodinger_disipativa_par}
$\mathrm{Re}(\Phi(t,u_{0}))$ and $\mathrm{Im}(\Phi(t,u_{0}))$
for $t=0$ (left), $t=2$ (center) and $t=10$ (right)}
\end{center}
\end{figure}

\bibliographystyle{plain}
\bibliography{references}

\begin{thebibliography}{10}

\bibitem{Aloui2008a}
L.~Aloui.
\newblock Smoothing effect for regularized {S}chr\"odinger equation on bounded
  domains.
\newblock {\em Asymptot. Anal.}, 59(3-4):179--193, 2008.

\bibitem{Aloui2008}
L.~Aloui.
\newblock Smoothing effect for regularized {S}chr\"odinger equation on compact
  manifolds.
\newblock {\em Collect. Math.}, 59(1):53--62, 2008.

\bibitem{Aloui2013}
L.~Aloui, M.~Khenissi, and G.~Vodev.
\newblock Smoothing effect for the regularized {S}chr\"odinger equation with
  non-controlled orbits.
\newblock {\em Comm. Partial Differential Equations}, 38(2):265--275, 2013.

\bibitem{Aranson2002}
I.~Aranson and L.~Kramer.
\newblock The world of the complex {G}inzburg-{L}andau equation.
\newblock {\em Rev. Modern Phys.}, 74(1):99--143, 2002.

\bibitem{Besse2002}
C.~Besse, B.~Bid{\'e}garay, and S.~Descombes.
\newblock Order estimates in time of splitting methods for the nonlinear
  {S}chr\"odinger equation.
\newblock {\em SIAM J. Numer. Anal.}, 40(1):26--40 (electronic), 2002.

\bibitem{Borgna2015}
J.~P. Borgna, M.~De Leo, D.~Rial, and C.~S. de~la Vega.
\newblock General splitting methods for abstract semilinear evolution
  equations.
\newblock {\em Commun. Math. Sci.}, 13(1):83--101, 2015.

\bibitem{Castella2009}
F.~Castella, P.~Chartier, S.~Descombes, and G.~Vilmart.
\newblock Splitting methods with complex times for parabolic equations.
\newblock {\em BIT}, 49(3):487--508, 2009.

\bibitem{Cazenave2003}
T.~Cazenave.
\newblock {\em Semilinear {S}chr\"odinger equations}, volume~10 of {\em Courant
  Lecture Notes in Mathematics}.
\newblock New York University Courant Institute of Mathematical Sciences, New
  York, 2003.

\bibitem{Chin2010}
S.~Chin.
\newblock Multi-product splitting and {R}unge-{K}utta-{N}ystr\"om integrators.
\newblock {\em Celestial Mech. Dynam. Astronom.}, 106(4):391--406, 2010.

\bibitem{Dekker1984}
K.~Dekker and J.~G. Verwer.
\newblock {\em Stability of {R}unge-{K}utta methods for stiff nonlinear
  differential equations}, volume~2 of {\em CWI Monographs}.
\newblock North-Holland Publishing Co., Amsterdam, 1984.

\bibitem{Descombes2010}
S.~Descombes and M.~Thalhammer.
\newblock An exact local error representation of exponential operator splitting
  methods for evolutionary problems and applications to linear {S}chr\"odinger
  equations in the semi-classical regime.
\newblock {\em BIT}, 50(4):729--749, 2010.

\bibitem{Descombes2013}
St{\'e}phane Descombes and Mechthild Thalhammer.
\newblock The {L}ie-{T}rotter splitting for nonlinear evolutionary problems
  with critical parameters: a compact local error representation and
  application to nonlinear {S}chr\"odinger equations in the semiclassical
  regime.
\newblock {\em IMA J. Numer. Anal.}, 33(2):722--745, 2013.

\bibitem{Gauckler2011}
L.~Gauckler.
\newblock Convergence of a split-step {H}ermite method for the
  {G}ross-{P}itaevskii equation.
\newblock {\em IMA J. Numer. Anal.}, 31(2):396--415, 2011.

\bibitem{Goldman1996}
D.~Goldman and T.~J. Kaper.
\newblock {$N$}th-order operator splitting schemes and nonreversible systems.
\newblock {\em SIAM J. Numer. Anal.}, 33(1):349--367, 1996.

\bibitem{Hairer1993}
E.~Hairer, S.~P. N{\o}rsett, and G.~Wanner.
\newblock {\em Solving ordinary differential equations. {I}}, volume~8 of {\em
  Springer Series in Computational Mathematics}.
\newblock Springer-Verlag, Berlin, second edition, 1993.
\newblock Nonstiff problems.

\bibitem{Hansen2009}
E.~Hansen and A.~Ostermann.
\newblock High order splitting methods for analytic semigroups exist.
\newblock {\em BIT}, 49(3):527--542, 2009.

\bibitem{Kopell1973}
N.~Kopell and L.~N. Howard.
\newblock Plane wave solutions to reaction-diffusion equations.
\newblock {\em Studies in Appl. Mat.}, 52:291--328, 1973.

\bibitem{Lubich2008}
C.~Lubich.
\newblock On splitting methods for {S}chr\"odinger-{P}oisson and cubic
  nonlinear {S}chr\"odinger equations.
\newblock {\em Math. Comp.}, 77(264):2141--2153, 2008.

\bibitem{Neri1987}
F.~Neri.
\newblock Lie algebras and canonical integration.
\newblock {\em Department of Physics report, University of Maryland}, 1987.

\bibitem{Reed1975}
Michael Reed and Barry Simon.
\newblock {\em Methods of modern mathematical physics. {II}. {F}ourier
  analysis, self-adjointness}.
\newblock Academic Press [Harcourt Brace Jovanovich, Publishers], New
  York-London, 1975.

\bibitem{Ruth1983}
R.~D. Ruth.
\newblock A canonical integration technique.
\newblock {\em IEEE Transactions on Nuclear Science}, NS 30(4):2669--2671,
  1983.

\bibitem{Sherratt2003}
J.~A. Sherratt.
\newblock Periodic travelling wave selection by {D}irichlet boundary conditions
  in oscillatory reaction-diffusion systems.
\newblock {\em SIAM J. Appl. Math.}, 63(5):1520--1538 (electronic), 2003.

\bibitem{Tadmor1986}
E.~Tadmor.
\newblock The exponential accuracy of {F}ourier and {C}hebyshev differencing
  methods.
\newblock {\em SIAM J. Numer. Anal.}, 23(1):1--10, 1986.

\bibitem{Yoshida1990}
H.~Yoshida.
\newblock Construction of higher order symplectic integrators.
\newblock {\em Phys. Lett. A}, 150:262--268, 1990.

\end{thebibliography}

\end{document}